\documentclass[1996/10/24]{amsart}
\usepackage{color}
\usepackage{amssymb}

\newtheorem{theorem}{Theorem}[section]
\newtheorem{lemma}[theorem]{Lemma}
\newtheorem{proposition}[theorem]{Proposition}

\theoremstyle{definition}
\newtheorem{definition}[theorem]{Definition}

\theoremstyle{remark}
\newtheorem{remark}[theorem]{Remark}

\numberwithin{equation}{section}

\def\Real{\mathbb{R}}

\def\Phim{\mathrm{\Phi}}

\def\Gm{\mathrm{G}}

\def\Lm{\mathrm{L}}

\def\Pm{\mathrm{P}}

\def\bfzero{\mathit{\textbf{0}}}
\def\bff{\textit{\textbf{f}}}

\def\bfn{\textit{\textbf{n}}}

\def\bfu{\textit{\textbf{u}}}
\def\bfv{\textit{\textbf{v}}}
\def\bfw{\textit{\textbf{w}}}

\def\bfE{\textit{\textbf{E}}}
\def\bfL{\textit{\textbf{L}}}
\def\bfH{\textit{\textbf{H}}}
\def\bfM{\textit{\textbf{M}}}

\def\bfP{\textit{\textbf{P}}}

\def\bfV{\textit{\textbf{V}}}
\def\bfW{\textit{\textbf{W}}}

\def\uhat{\widehat{\bfu}}

\def\itPi{\mathit{\Pi}}

\def\uhat{\widehat{\bfu}}
\def\vhat{\widehat{\bfv}}
\def\what{\widehat{\bfw}}

\def\phihat{\widehat{\bfphi}}

\def\bfY{\boldsymbol{Y}}

\def\bfmu{\boldsymbol{\mu}}

\def\bftheta{\boldsymbol{\theta}}

\def\bftheta{\boldsymbol{\theta}}
\def\bfphi{\boldsymbol{\phi}}
\def\bfpsi{\boldsymbol{\psi}}

\def\calF{\mathcal{F}}

\def\calO{\mathcal{O}}

\def\calT{\mathcal{T}}

\def\calE{\mathcal{E}}

\def\nrm{|\hspace{-0.05cm}|\hspace{-0.05cm}|}

\begin{document}
\title[A superconvergent  HDG for the incompressible Navier-Stokes equations]
{A superconvergent HDG method for the Incompressible\\ 
Navier-Stokes Equations on general polyhedral meshes}


\author{Weifeng Qiu}
\address{Department of Mathematics, City University of Hong Kong,
83 Tat Chee Avenue, Kowloon, Hong Kong, China}
\email{weifeqiu@cityu.edu.hk}
\thanks{
Weifeng Qiu is partially supported by a grant from the Research Grants 
Council of the Hong Kong Special Administrative Region, China 
(Project No. CityU 11302014). As a convention the names of the authors are alphabetically ordered. Both authors contributed equally in this article. 
}

\author{Ke Shi}
\address{Department of Mathematics $\&$ Statistics, Old Dominion University, Norfolk, VA 23529, USA}
\curraddr{}
\email{kshi@odu.edu}


\begin{abstract}
We present a superconvergent hybridizable discontinuous Galerkin (HDG) method for the steady-state 
incompressible Navier-Stokes equations on general polyhedral meshes. For arbitrary conforming polyhedral mesh, we use polynomials of degree $k + 1$, $k$, $k$ to approximate 
the velocity, velocity gradient and pressure, respectively. In contrast, we only use polynomials of degree $k$ 
to approximate the numerical trace of the velocity on the interfaces. Since the numerical trace of the velocity 
field is the only globally coupled unknown, this scheme allows a very efficient implementation of the method.
For the stationary case, and under the usual smallness condition for the source term, 
we prove that the method is well defined and that the global $L^{2}$-norm of the error in each of 
the above-mentioned variables and the discrete $H^{1}$-norm of the error in the velocity
 converge with the order of $k + 1$ for $k \geq 0$. We also show that for $k\geq 1$, 
the global $L^{2}$-norm of the error in velocity converges with the order of $k+2$. 
From the point of view of degrees of freedom of the globally coupled unknown: numerical trace, 
this method achieves optimal convergence for all the above-mentioned variables in $L^{2}$-norm for $k\geq 0$, 
superconvergence for  the velocity in the discrete $H^{1}$-norm without postprocessing for $k\geq 0$,   
and superconvergence for  the velocity in $L^{2}$-norm without postprocessing for $k\geq 1$. \end{abstract}

\maketitle

\section{Introduction}

In this paper, we consider a new hybridizable discontinuous Galerkin (HDG) method for the steady-state 
incompressible Navier-Stokes equations, which can be written as the following first order system:
\begin{subequations}\label{eq:NSE}
\begin{eqnarray}
 \Lm=\nabla \bfu&&\mbox{ in }\Omega,\label{eq:NSE1}\\
 -\nu\nabla\cdot \Lm+\nabla\cdot(\bfu\otimes\bfu)+\nabla p=\bff &&\mbox{ in }\Omega,\label{eq:NSE2}\\
 \nabla \cdot \bfu =0 &&\mbox{ in }\Omega,\label{eq:NSE3}\\
 \bfu=\bfzero &&\mbox{ on }\partial\Omega,\label{eq:NSE4}\\
 \int_{\Omega}p=0,\label{eq:NSE5}
\end{eqnarray}
where the unknowns are the velocity $\bfu$ , the pressure $p$, and the gradient of the velocity $\Lm$. $\nu$ is 
the kinematic viscosity and $\bff\in \boldsymbol{L}^2(\Omega)$ is 
the external body force. The domain $\Omega\subset \Real^d$, $d= 2, 3$ is polygonal ($d=2$) or polyhedral ($d=3$).
\end{subequations}

To define the HDG method, we adopt the notations and norms used in \cite{CCQ2015}. 
We consider conforming triangulation $\calT_h$ of $\Omega$ made of shape-regular \emph{polyhedral elements} which can be non-convex. 
We denote by $\calE_h$ the set of all faces $F$ of all elements $K\in \calT_h$ 
and set $\partial \calT_h:=\{\partial K:K\in \calT_h\}$.
For scalar-valued functions $\phi$ and $\psi$, we write
\[
 (\phi,\psi)_{\calT_h}:=\sum_{K\in\calT_h}(\phi,\psi)_K, \, \, \langle\phi,\psi\rangle_{\partial\calT_h}
 :=\sum_{K\in\calT_h}\langle\phi,\psi \rangle_{\partial K}.
\]
Here $(\cdot,\cdot)_D$ denotes the integral over
the domain $D\subset \Real^d$, and $\langle \cdot,\cdot \rangle_D$ denotes the
integral over $D\subset \Real^{d-1}$. For vector-valued and
matrix-valued functions, a similar notation is taken. For example, for
vector-valued functions, we write
$(\bfphi,\bfpsi)_{\calT_h}:=\sum_{i=1}^n(\phi_i,\psi_i)_{\calT_h}$. 
For matrix-valued functions, we write 
$(\phi, \psi)_{\calT_h}:= \sum_{1\leq i,j\leq n}(\phi_{ij}, \psi_{ij})_{\calT_h}$.

Like all other HDG schemes, to define the HDG method for 
the problem, we introduce an additional unknown \emph{numerical trace} which is the approximation of the velocity on 
the skeleton of the mesh. Namely, our HDG method seeks an approximation 
$(\Lm_h, \bfu_h, p_h,\uhat_h)\in \Gm_h\times \bfV_h\times Q_h\times
\bfM_h^{\,0}$ to the exact solution 
$(\Lm|_{\mathcal{T}_h}, \bfu|_{\mathcal{T}_h}, p|_{\mathcal{T}_h},\bfu|_{\mathcal{E}_h})$
where the finite dimensional spaces are defined as:
\begin{alignat*}{3}
 \Gm_h&:=\{\Gm\in {\mathrm{L}}^2(\Omega):&&\;\Gm|_K\in \Pm_k(K), &&\,\, \forall K\in \calT_h\},\\
 \bfV_h&:=\{\bfv\in \boldsymbol{L}^2(\Omega):&&\;\bfv|_K\in \bfP_{k+1}(K), &&\,\, \forall K\in \calT_h\},\\
 Q_h&:=\{p\in L_{0}^2(\Omega):&&\;p|_K\in P_k(K), &&\,\, \forall K\in \calT_h\},\\
 \bfM_h&:=\{\bfmu\in \boldsymbol{L}^2(\calE_h):\;&&\bfmu|_F\in \bfP_k(F), &&\,\, \forall F\in \calE_h\},\\
 \bfM_h^{\,0}&:=\{\bfmu\in \bfM_h:\;&&\bfmu|_{\partial \Omega}=0\}.
\end{alignat*}
Here $P_l(D)$ denotes the set of polynomials of total degree at most $l \geq 0$
defined on $D$, $\bfP_k(D)$ denotes the set of vector-valued functions whose $d$
components lie in $P_k(D)$, $\Pm_k(D)$ denotes the set of square
matrix-valued functions whose $d\times d$ entries also lie in $P_k(D)$, 
and $L_{0}^{2}(\Omega) = \{p\in L^{2}(\Omega): \int_{\Omega}p = 0\}$.

The method determines the approximate solution by requiring that it solves the
following weak formulation:
\begin{subequations}\label{eq:NSEHDG}
\begin{align}
 (\Lm_h, \Gm)_{\calT_h}+(\bfu_h, \nabla \cdot \Gm)_{\calT_h}-\langle \uhat_h, 
 \Gm\bfn\rangle_{\partial \calT_h}&=0,\label{eq:NSEHDGI}\\
(\nu\Lm_h, \nabla\bfv)_{\calT_h}-(\bfu_h\otimes \bfu_h ,\nabla \bfv)_{\calT_h}
-(p_h, \nabla \cdot\bfv)_{\calT_h}\label{eq:NSEHDGII}
-\langle \nu\,\widehat{\Lm}_h\bfn-\widehat{p}_h\bfn-(\widehat{\bfu}_h\otimes \uhat_h)\bfn,  
\bfv\rangle_{\partial\calT_h}& \\
\nonumber
-  ( \frac12 (\nabla \cdot \bfu_h) \, \bfu_h, \bfv)_{\calT_h} + \langle \frac12 
(\bfu_h \otimes (\bfu_h - \widehat{\bfu}_h)) \bfn, \bfv \rangle_{\partial \calT_h} 
&=(\bff, \bfv)_{\calT_h} ,\\
 -(\bfu_h, \nabla q)_{\calT_h}+\langle\uhat_h\cdot \bfn,q
 \rangle_{\partial\calT_h}&=0,\label{eq:NSEHDGIII}\\
 \langle \nu\,\widehat{\Lm}_h\bfn-\widehat{p}_h\bfn -(\widehat{\bfu}_h
 \otimes \widehat{\bfu}_h)\bfn, 
 \bfmu\rangle_{\partial \calT_h}&=0\label{eq:NSEHDGIV},
\end{align}
for all $(\Gm,\bfv, q,\bfmu)\in \Gm_h\times \bfV_h\times Q_h\times
\bfM_h^{\,0}$. Here,
\begin{alignat}{2}\label{numerical_flux}
(\nu\,\widehat{\Lm}_h-\widehat{p}_h)\bfn:=&\;\nu\Lm_h\bfn-p_h\bfn 
- \frac{\nu}{h}(\itPi_M \bfu_h-\uhat_h) - \tau_C(\uhat_h) (\bfu_h -\uhat_h)
&&\quad\mbox{ on }\partial\mathcal{T}_h,
\\
\label{stab_op}
\tau_C(\uhat_h):=&\;\max(\uhat_h \cdot \bfn,0) 
&&\quad\mbox{ on each  } F \in \partial\mathcal{T}_h.
\end{alignat}
\end{subequations}
Here $\itPi_M$ is the $L^2-$projection onto $\bfM_h$. The goal of  this paper is to consider the 
analytical aspects of the method including the rigorous proof of the uniqueness and existence 
of the solution of the above system and the error estimates for all unknowns. The 
computational aspects of the method will be discussed in a separate paper.

Like other HDG schemes, the HDG method (\ref{eq:NSEHDG}) uses the numerical trace 
of the primary variable  $\uhat_h$ 
as the only globally-coupled variable. Our formulation is close to that 
of the HDG method in \cite{CCQ2015, NPCockburn2011}, in which they have 
the same global degrees of freedom  $\uhat_h$. 
Nevertheless, there are three crucial differences which lead to special properties of 
our HDG method. Firstly, our method uses $\bfP_{k+1}(\mathcal{T}_{h})$ 
to approximate the primary variable, which is the velocity, on each element 
while methods in \cite{CCQ2015, NPCockburn2011} uses $\bfP_{k}(\mathcal{T}_{h})$ 
instead. Secondly, the diffusion part of the stabilization function 
$\frac{\nu}{h}(\itPi_M \bfu_h-\uhat_h)$ in (\ref{numerical_flux}) 
is totally different from those used in \cite{CCQ2015, NPCockburn2011}. Finally, 
motivated by the work in \cite{WalugaThesis} and \cite{CockburnKanschatSchoetzauNS05},
we insert two terms $- (\frac12 (\nabla \cdot \bfu_h) \, 
\bfu_h, \bfv)_{\calT_h}$ and $\langle \frac12 (\bfu_h \otimes 
(\bfu_h - \widehat{\bfu}_h)) \bfn, \bfv \rangle_{\partial \calT_h}$ 
in \eqref{eq:NSEHDGII}. As a consequence of the above modifications, the new HDG method
allows to use general polygonal mesh with optimal approximations for all unknowns. In addition, 
from the implementation point of view, like many other numerical methods for Navier-Stokes equation,
we apply the classical Picard iteration to obtain the numerical solution. In each iteration we need to solve
an Oseen equation. Due to our modification in \eqref{eq:NSEHDGII}, we can use the convection field 
obtained from previous step directly without the use of postprocessing. We notice that by adding these 
two terms in \eqref{eq:NSEHDGII}, the HDG method 
(\ref{eq:NSEHDG}) is not locally conservative 
(see \cite{CockburnKanschatSchoetzauNS05} 
for detailed explanation). On the other hand, since 
the HDG method (\ref{eq:NSEHDG}) has high order 
accuracy for the approximations to all variables, 
lack of being locally conservative is only a minor issue. 

It is worth to mention that the HDG method using an enhanced space for 
the primary variable and the special stabilization function like 
$\frac{\nu}{h}(\itPi_M \bfu_h-\uhat_h)$ was first introduced by Lehrenfeld in 
Remark 1.2.4 for diffusion problem in \cite{Lehrenfeld}. He numerically showed 
that the methods provide optimal order of convergence for all unknowns without analysis.
In \cite{QiuShi2015a}, we gave rigorous analysis for this approach for linear elasticity problems. 
Optimal order of convergence for all unknowns is obtained for both equations. 
In \cite{Oikawa14_diffusion}, Oikawa analyzed a HDG method for diffusion problem which 
uses the same polynomial spaces as in \cite{Lehrenfeld}, with a different choice of 
the numerical flux, he proved the optimality of the method for all unknowns. 
Since the polynomial order of the numerical trace of these HDG methods 
\cite{Lehrenfeld,QiuShi2015a,Oikawa14_diffusion} is one less then 
that of the approximation space of the primary variable, from the point of view of degrees 
of freedom of the globally coupled unknown, they obtain superconvergence for the 
primary variable without postprocessing. In addition, all these methods work on 
general polyhedral meshes. However, the standard stability analysis for these methods 
can only provide the upper bound of 
\begin{align*}
\big(\Vert \nabla\bfu_h\Vert_{\mathcal{T}_{h}}^{2}+h^{-1}
\Vert \itPi_M \bfu_h-\uhat_h\Vert^2_{\partial\mathcal{T}_{h}}\big)^{1/2},
\end{align*}
which can not control the standard discrete $H^{1}$-norm of $\bfu_h$.
We would like to emphasize that the control of the standard discrete 
$H^{1}$-norm of $\bfu_h$ is essentially necessary in the proof of 
the HDG method (\ref{eq:NSEHDG}) having a unique solution. 
In \cite{QiuShi2015b}, roughly speaking, we prove that 
\begin{align}
\label{magic_ineq}
\big(\Vert \nabla\bfu_h\Vert_{\mathcal{T}_{h}}^{2}+h^{-1}
\Vert \bfu_h-\uhat_h\Vert^2_{\partial\mathcal{T}_{h}}\big)^{1/2}
\leq C \big(\Vert \nabla\bfu_h\Vert_{\mathcal{T}_{h}}^{2}+h^{-1}
\Vert \itPi_M \bfu_h-\uhat_h\Vert^2_{\partial\mathcal{T}_{h}}\big)^{1/2}.
\end{align}
The inequality (\ref{magic_ineq}) enables us to control the standard discrete 
$H^{1}$-norm of $\bfu_h$ of the HDG method (\ref{eq:NSEHDG}) such that 
we can show our method has a unique solution under the usual smallness condition 
for the source term.

In literature, there are many existing mixed and DG methods designed for Navier-Stokes equations. 
See the classic mixed methods \cite{GiraultRaviart86,BrezziFortin91,Fortin93}, the stabilized methods 
proposed in \cite{HughesFrancaBalestraV86,HughesFrancaCFDVII87,KechkarSilvester92} and the DG methods 
\cite{BakerJureidiniKarakashian90,Karakashian,CockburnKanschatSchoetzauSchwabStokes02,GiraultRiviereWheelerNavierStokes05,
Toselli02,ShoetzauSchwabToselli03,CarreroCockburnSchoetzau06,CockburnKanschatSchoetzauNS05,CockburnKanschatSchoetzauDIV07,
 CKS2009,MontlaurFernandezHuertaDGHdiv08}. An IP-like method and a compact discontinuous Galerkin (CDG) 
 method were introduced in \cite{montlaur2010discontinuous}. Recently, 
 In \cite{CaiWangZhang2010}, a mixed method is developed 
based on the pseudostress-velocity formulation. 
This method uses row-wise $k$-th order RT space and  $\bfP_{k}(\mathcal{T}_{h})$ 
to approximate the pseudostress and velocity, respectively. 
The method provides the convergence of both variables in $L^{3}$ norm of order ${k+1-\frac{d}{6}}$.  
In \cite{CaiZhang2012}, a modified method is introduced and 
it approximates the pressure directly with same convergence rate as in \cite{CaiWangZhang2010}.
Later in \cite{HowellWalkington2013}, a new mixed finite element method is introduced 
in which the stress is the primary variable. This method uses 
$\Pm_{k}(\mathcal{T}_{h})$, row-wise $k$-th order RT space 
and $\bfP_{k}(\mathcal{T}_{h})$ to approximate the velocity 
gradient, stress and velocity, respectively. 
The convergence of the velocity gradient and velocity in $L^{2}$-norm 
and the stress in $H(\text{div})$-norm is of order ${k}$. More recently, in 2015, Cockburn et al \cite{CCQ2015} gave an error analysis of the  
HDG method developed in \cite{NPCockburn2011} which is  close to method in this paper. Our method may be criticized by the fact that with the modification of the scheme, we no longer have the local conservation of the momentum. 
Nevertheless, our approach has several advantages comparing with the one in  \cite{CCQ2015, NPCockburn2011}. 
For instance, the analysis in \cite{CCQ2015, NPCockburn2011} is only valid for simplicical meshes and it needs a postprocessing 
procedure to obtain superconvergent approximation to the velocity. 
The method  in \cite{CCQ2015, NPCockburn2011} does not have superconvergent approximation to 
the velocity in the discrete $H^{1}$-norm if $k=0$, while our method does even this the lowest order case.
From the implementation point of view, in each iteration, 
the scheme in \cite{CCQ2015, NPCockburn2011} needs to solve a Oseen equation using a postprocessed convection field from the previous iteration while in our scheme we can use the velocity field directly from the previous iteration without any postprocessing.

In this paper, we prove that the discrete $H^{1}$-norm of the error in the velocity,
the $L^2$-norm of the error in the velocity, the pressure and even in the velocity 
gradient converge with the order $k+1$ for any $k\ge0$; and that the velocity, 
for $k\geq 1$, converges with order $k + 2$. Notice that as a built-in feature of 
HDG methods, see \cite{CockburnQiuShi11}, the degrees of freedom of the globally-coupled unknown 
comes from the numerical trace of the velocity on the mesh skeleton. 
From the point of view of the global degrees of freedom, the method provides 
optimal convergent approximations to the velocity, velocity gradient and pressure in $L^{2}$-norm 
for $k\geq 0$, superconvergent approximation to the velocity in the discrete $H^{1}$-norm without 
postprocessing for $k\geq 1$, and superconvergent approximation to the velocity in $L^{2}$ norm without 
postprocessing for $k\geq 1$. In addition, the analysis of our method is valid for general polyhedral meshes. 
To the best of our knowledge, no other known finite element method for the Navier-Stokes equations has all of these properties. 

The rest of paper is organized as follows. In Section 2, we introduce our HDG method for the problem and present the main 
a priori error estimates. In Section 3, we present some preliminary inequalities and stability estimates. In Section 4, we prove 
the existence and uniqueness of the numerical solution. In Section 5, we provide the detailed proof of the main results.

\section{Main Results}

In this section, we present the main error estimates results.
To state the main results, we need to introduce some notations.
We use the standard definitions for the Sobolev spaces $W^{\ell,p}(D)$ 
for a given domain $D$ with norm 
\[
 \|\phi\|_{\ell,p,D}=(\sum_{|\alpha|\leq \ell}\|D^{\alpha}\phi\|_{0,p,D}^p)^{1/p}.
\] 
For vector- and matrix-valued functions $\bfphi$ and $\Phim$, we use $
 \|\bfphi\|_{\ell,p,D}=\sum_{i=1}^d\|\bfphi_i\|_{\ell,p,D}$, and
 $\|\Phim\|_{\ell,p,D}=\sum_{i,j=1}^d\|\Phim_{ij}\|_{\ell,p,D}.$ Moreover, 
when $p=2$ and $\ell<\infty$, we denote $W^{\ell,2}(D)$ by $H^\ell(D)$ and $\|\cdot\|_{\ell,2,D}$, 
by $\|\cdot\|_{\ell,D}$. When $l = 0,$ we denote $W^{0,p}(D)$ by $L^p(D)$ and the norm by $\|\cdot\|_{L^p(D)}$, 
when $\ell=0$ and $p=2$, we denote the $L^2(D)$ norm by $\|\cdot\|_D$. 

We also introduce the following norms and seminorms:
\begin{alignat*}{2}
\nrm(\bfv,\bfmu)\nrm_{0,h}
&:=\Big(\|\bfv\|_{\calT_h}^2+ \,(\| h_{K}^{1/2} \bfmu\|^2_{\partial \calT_h}+\| h_{K}^{1/2}(\bfv-\bfmu)\|_{\partial \calT_h}^2)\Big)^{1/2}
&&\quad\mbox{$\forall\;(\bfv,\bfmu)$ in $\boldsymbol{H}^1(\mathcal{T}_h)\times
\boldsymbol{L}^2(\mathcal{E}_h)$},
\\
\nrm(\bfv,\bfmu)\nrm_{1,h}
&:=\Big(\|\nabla \bfv\|_{\calT_h}^2+ \|h_{K}^{-1/2} (\bfv-\bfmu)\|_{\partial \calT_h}^2\Big)^{1/2}
&&\quad\mbox{$\forall\;(\bfv,\bfmu)$ in $\boldsymbol{H}^1(\mathcal{T}_h)\times
\boldsymbol{L}^2(\mathcal{E}_h)$},
\\
\nrm(\bfv,\bfmu)\nrm_{\infty,h}
&:=\|\bfv\|_{L^\infty(\Omega)}+\|\bfmu\|_{L^\infty(\mathcal{E}_h)}
&&\quad\mbox{$\forall\;(\bfv,\bfmu)$ in $\boldsymbol{L}^\infty(\Omega)\times
\boldsymbol{L}^\infty(\mathcal{E}_h)$}.
\end{alignat*}
Here $\|\cdot\|_{\partial \mathcal{T}_h} := \big( \sum_{K \in \mathcal{T}_h} \|\cdot\|_{\partial K}^{2}\big)^{1/2}$.
We also set 
\[
\|\bfv\|_{0,h}:=\Vert \bfv\Vert_{L^{2}(\Omega)},\quad \|\bfv\|_{1,h}:=\nrm(\bfv,\{\!\!\{ \bfv \}\!\!\})\nrm_{1,h},
\]
where the average of $\bfv$, $\{\!\!\{ \bfv \}\!\!\}$, is defined as follows: 
On an interior face $F=\partial K^-\cap \partial K^+$, we have
$
\{\!\!\{ \bfv\}\!\!\}:=\frac12(\bfv^+ + \bfv^-),
$
where $\bfv^\pm$ denote the trace of $\bfv$ from the interior of
$K^\pm$ and $\boldsymbol{n}^\pm$ is the outward unit normal to $K^\pm$. On 
a boundary face $F\subset\partial K^{-}\cap \partial \Omega$, we formally set 
$\bfv^+:=\bfv$ such that $\{\!\!\{ \bfv\}\!\!\} = \bfv$ on $\partial \Omega$. 
We note that $\|\cdot\|_{1,h}$ is the standard discrete $H^{1}$-seminorm.

We are now ready to state our first main result on the existence and uniqueness of the numerical solution. 
\begin{theorem}[Existence, uniqueness and stability]\label{thm:existence}
If $\|\bff\|_{\Omega}$ is small enough,
the HDG method \eqref{eq:NSEHDG} has a unique solution $(\Lm,
\bfu_h,p_h,\uhat_h)\in \Gm_h\times \bfV_h\times Q_h\times \boldsymbol{M}_h^{\,0}$. 
Furthermore, the following stability bound is satisfied
\begin{align}\label{stab}
\nrm(\bfu_h,\uhat_h)\nrm_{1,h}&\leq C\nu^{-1}\|\bff\|_{\Omega}.
\end{align}
for some constant $C$ independent of $\nu$, the discretization parameters and the exact solution.
\end{theorem}

Next we present the error estimates result for all unknowns. In order to have optimal $L^2-$error estimate for the velocity, 
we need some regularity assumption of the following dual problem. Consider the problem of seeking $(\bfphi,\psi)$ such that
\begin{subequations}
\label{eq:dualNSE}
\begin{align}
\Phim-\nabla\bfphi&=0& \text{ in }\Omega,\label{eq:dualNSE1}\\
-\nu \nabla \cdot\Phim -\nabla\cdot(\bfphi\otimes \bfu)- \nabla \psi  - \frac12 (\nabla \bfphi)^{\top}\bfu 
+ \frac12 (\nabla \bfu)^{\top} \bfphi 
& =\bftheta & \text{ in }\Omega,\label{eq:dualNSE2}\\
\nabla\cdot \bfphi & = 0 & \text{ in }\Omega,\label{eq:dualNSE3}\\
\bfphi & = 0 & \text{ on }\partial\Omega.
\end{align}
\end{subequations}
%
Assume that the solution to the dual problem satisfies the following regularity estimate:
\begin{equation}\label{eq:dualreg}
\|\Phim\|_{1,\Omega}+\|\bfphi\|_{2,\Omega}+\|\psi\|_{1,\Omega}\leq C_r\|\bftheta\|_{\Omega}.
\end{equation}

\begin{remark} 
If $\Vert \bfu\Vert_{H^{1}(\Omega)}$ is small enough compared with the diffusion coefficient $\nu$, 
the dual problem (\ref{eq:dualNSE}) has a unique solution $(\bfphi, \psi)\in \boldsymbol{H}_{0}^{1}(\Omega)\times 
H^{1}(\Omega)/\mathbb{R}$. In fact, when we use the standard energy argument, we need to have 
\begin{align}
\label{bound_term_dual}
\frac{1}{2}\vert ((\nabla \bfu)^{\top} \bfphi -(\nabla \bfphi)^{\top}\bfu,\bfphi)_{\Omega}\vert
\leq \nu\Vert \nabla \bfphi\Vert_{\Omega}^{2}
\end{align}
to obtain energy estimate of $\bfphi$. 
 We notice that $\frac{1}{2}\vert 
 ((\nabla \bfu)^{\top} \bfphi 
-(\nabla \bfphi)^{\top}\bfu,\bfphi)_{\Omega}\vert
\leq C \Vert \bfu\Vert_{H^{1}(\Omega)}\Vert \bfphi\Vert_{H^{1}(\Omega)}^{2}$. 
It is easy to see that (\ref{bound_term_dual}) 
holds if $\Vert \bfu\Vert_{H^{1}(\Omega)}$ 
is small enough compared with the diffusion coefficient $\nu$.
This completes the proof of the above claim. If we further assume 
$\bfu \in \bfW^{1, 3}(\Omega)\cap \bfL^{\infty}(\Omega)$, 
then, the regularity assumption (\ref{eq:dualreg}) 
comes from a standard regularity estimate \cite{GiraultRaviart86} for the Stokes equations.
\end{remark}

Now we are ready to present our second and main result:

\begin{theorem}\label{thm:error}
If $\|\bff\|_{\Omega}$ is small enough, then we have
\[
\|\Lm - \Lm_h\|_{\Omega} + \|\bfu - \bfu_h\|_{\Omega} + \|\bfu - \bfu_{h}\|_{1,h} + \|p - p_h\|_{\Omega} \le \mathcal{C} h^{k+1}, 
\]
Here the constant $\mathcal{C}$ depends on $\|\bfu\|_{L^{\infty}(\Omega)}, \|\bfu\|_{k+2, \Omega}, \|p\|_{k+1, \Omega}, \nu$ and $k$.
In addition, if the regularity assumption \eqref{eq:dualreg} holds and $\bfu \in \bfW^{1, \infty}(\Omega)$, then for $k \ge 1$ we have
\[
\|\bfu - \bfu_h\|_{\Omega} \le \mathcal{C}_D h^{k+2}.
\]
Here $\mathcal{C}_D$ depends on $\|\bfu\|_{L^{\infty}(\Omega)}, \|\bfu\|_{k+2, \Omega}, \|p\|_{k+1, \Omega}, \nu$ and $k, C_r$.
\end{theorem}

\section{Preliminary estimates}
In this section, we present some preliminary inequalities for the proof of our main results. First, we would like to recall 
an important inequality which was introduced in \cite{QiuShi2015b}. Here we write it in a slightly general way. 
Though our results in this section and the following ones are valid for conforming meshes 
with shape regular assumption, we assume the meshes are quasi-uniform for the sake of simplicity.

\begin{lemma}\label{H1_ineq}
For any given function $(\Lm, \bfv, \bfmu) \in \Gm_h\times \bfV_h \times \bfM_h$ satisfying \eqref{eq:NSEHDGI}, then we have
\[
\nrm(\bfv,\bfmu)\nrm_{1,h} \le C_{\text{HDG}}( \|\Lm\|_{\Omega} + h^{-\frac12}\|\itPi_M \bfv - \bfmu\|_{\partial \mathcal{T}_h}).
\]
\end{lemma}
For the proof of the above result, we refer the Lemma 3.2 in \cite{QiuShi2015b}. In addition, we also need the following basic inequalities:

\begin{lemma}\label{Lq_ineq}
For $1 \le q < \infty \,  (d = 2)$, $1 \le q \le 4 \, (d =3)$, there exist positive constant $C_q$ such that 
\begin{subequations}
\begin{align}\label{basic_1}
\|\bfv\|_{L^q(\Omega)} &\le C_q \|\bfv\|_{1, h}, \quad && \forall \; \bfv \in \bfV(h), \\
\label{basic_2}
\|\bfv\|_{L^q(\Omega)} &\le C_q \nrm(\bfv, \bfmu) \nrm_{1,h}, \quad && \forall \; (\bfv, \bfmu) \in \bfV(h) \times \bfM^0_h, \\
\intertext{Here $\bfV(h) := \bfH^1_0(\Omega) + \bfV_h$. In addition, we have a trace inequality: }
\label{basic_3}
 \|\bfv\|_{L^4(\partial \calT_h)} & \le C h^{-\frac14} \|\bfv\|_{1, h} \le C h^{- \frac14} 
 \nrm( \bfv, \bfmu) \nrm_{1,h}, \quad &&\forall \; (\bfv, \bfmu) \in \bfV(h) \times \bfM^0_h.
\end{align}
\end{subequations}
\end{lemma}

The proofs of \eqref{basic_1}-\eqref{basic_3} are provided in Proposition A.2 in \cite{CCQ2015}, 
Proposition 4.5 and (7.7) in \cite{Karakashian}. 
To simplify our notations, we group all the nonlinear terms in our formulation as the following operator:

\begin{definition}\label{nonlinear-O}
For any $(\bfw, \what), (\bfu, \uhat), (\bfv, \vhat) \in \bfH^1(\calT_h) \times \bfL^2(\calE_h)$, we define the operator:
\begin{align*}
\calO((\bfw, \what); (\bfu, \uhat), (\bfv, \vhat)) :=  &- ( \bfu \otimes \bfw , \nabla \bfv )_{\calT_h}
 - (\frac12 (\nabla \cdot \bfw) \bfu , \bfv )_{\calT_h} + \langle \frac12 \bfu \otimes (\bfw - \what) \bfn , \bfv \rangle_{\partial \calT_h} \\
&+ \langle \tau_C(\what) (\bfu - \uhat) , \bfv - \vhat \rangle_{\partial \calT_h} + \langle (\uhat \otimes \what) \bfn , \bfv - \vhat \rangle_{\partial \calT_h}.
\end{align*}
\end{definition}
The above operator plays a crucial rule in the analysis. It has the following coercive property:
\begin{proposition}\label{Ocoercive}
For any $(\bfw, \what), (\bfu, \uhat) \in \bfH^1(\calT_h) \times \bfL^2(\bfE_h)$, if $\uhat|_{\partial \Omega} = 0$, then we have
\[
\calO((\bfw, \what); (\bfu, \uhat), (\bfu, \uhat)) = \langle (\tau_C(\what) 
- \frac12 \what \cdot \bfn) (\bfu - \uhat), \bfu - \uhat \rangle_{\partial \calT_h} \ge 0.
\]
\end{proposition}

\begin{proof}
By Definition \ref{nonlinear-O}, we have
\begin{align*}
\calO((\bfw, \what); (\bfu, \uhat), (\bfu, \uhat)) :=  &- ( \bfu \otimes \bfw , \nabla \bfu )_{\calT_h} 
- (\frac12 (\nabla \cdot \bfw) \bfu , \bfu )_{\calT_h} + \langle \frac12 \bfu \otimes (\bfw - \what) \bfn , \bfu \rangle_{\partial \calT_h} \\
&+ \langle \tau_C(\what) (\bfu - \uhat) , \bfu - \uhat \rangle_{\partial \calT_h} + \langle (\uhat \otimes \what) \bfn , \bfu - \uhat \rangle_{\partial \calT_h}.
\end{align*}
Applying integration by parts for the first term, we have
\begin{align*}
( \bfu \otimes \bfw , \nabla \bfu )_{\calT_h} &= - (\nabla \cdot (\bfu \otimes \bfw), \bfu )_{\calT_h} 
+ \langle (\bfu \otimes \bfw) \bfn, \bfu \rangle_{\partial \calT_h} \\
& = - ((\nabla \cdot \bfw) \bfu, \bfu )_{\calT_h} - (\bfu \otimes \bfw, \nabla \bfu )_{\calT_h}  + \langle (\bfu \otimes \bfw) \bfn, \bfu \rangle_{\partial \calT_h}.
\end{align*}
This implies that 
\[
- (\bfu \otimes \bfw, \nabla \bfu )_{\calT_h} - (\frac12 (\nabla \cdot \bfw) \bfu , \bfu )_{\calT_h} 
+ \langle (\frac12 \bfu \otimes \bfw) \bfn, \bfu \rangle_{\partial \calT_h} = 0.
\]
Inserting above identity into $\calO((\bfw, \what); (\bfu, \uhat), (\bfu, \uhat))$, we have
\begin{align*}
\calO((\bfw, \what); (\bfu, \uhat), (\bfu, \uhat)) &= \langle \tau_C(\what) (\bfu - \uhat) , \bfu - \uhat \rangle_{\partial \calT_h} + \langle (\uhat \otimes \what) \bfn , \bfu - \uhat \rangle_{\partial \calT_h} -  \langle \frac12 (\bfu \otimes \what) \bfn , \bfu \rangle_{\partial \calT_h} \\
& = \langle (\tau_C(\what) - \frac12 \what \cdot \bfn) (\bfu - \uhat), \bfu - \uhat \rangle_{\partial \calT_h} - \langle \frac12 (\what \cdot \bfn) \uhat , \uhat \rangle_{\partial \calT_h} \\
& = \langle (\tau_C(\what) - \frac12 \what \cdot \bfn) (\bfu - \uhat), \bfu - \uhat \rangle_{\partial \calT_h} \ge 0.
\end{align*}
The last step is due to the fact that $\uhat$ is single valued on $\calE_h$ and $\uhat|_{\partial \Omega} = 0$.
\end{proof}

Next, we present a continuity result for the nonlinear operator $\mathcal{O}$ that we will use throughout the analysis. We first define the following space:
\[
\widetilde{\bfH_0^1}(\Omega) := \{(\bfw, \what) \in \bfH_0^1(\Omega) \times L^2(\calE_h) | \bfw|_{\calE_h} = \what\},
\]
The above space is the graph space of the trace mapping from $\bfH^1(\Omega)$ onto $\bfL^2(\calE_h)$. We are ready to state the following result:

\begin{lemma}\label{O-estimates}
There is a positive constant $C_{\mathcal{O}}$ such that
\begin{equation}\label{O_estimate1}
| \mathcal{O}((\bfw_1, \what_1); (\bfu, \uhat), (\bfv, \vhat)) - \mathcal{O}((\bfw_2, \what_2); (\bfu, \uhat), (\bfv, \vhat))|
 \le C_{\mathcal{O}} \nrm (\bfw_1, \what_1) - (\bfw_2, \what_2) \nrm_{1, h} \; \nrm (\bfu, \uhat) \nrm_{1, h} \;  \nrm (\bfv, \vhat) \nrm_{1, h}, 
\end{equation}
for all $(\bfw_1, \what_1), (\bfw_2, \what_2), (\bfu, \uhat) \in \widetilde{\bfH_0^1}(\Omega) +\big( \bfV_h \times \bfM^0_h \big)$ 
and any $(\bfv, \vhat) \in \bfV_h \times \bfM^0_h$.
\end{lemma}

\begin{proof}
Setting $\delta_{\bfw} := \bfw_1 - \bfw_2, \delta_{\what} := \what_1 - \what_2$, by the definition of the operator $\mathcal{O}$, we have
\begin{align*}
\mathcal{O}((\bfw_1, \what_1); (\bfu, \uhat), (\bfv, \vhat)) &- \mathcal{O}((\bfw_2, \what_2); (\bfu, \uhat), (\bfv, \vhat)) = \\
&- ( \bfu \otimes \delta_{\bfw} , \nabla \bfv )_{\calT_h} - (\frac12 (\nabla \cdot \delta_{\bfw}) \bfu , \bfv )_{\calT_h} + \langle \frac12 \bfu \otimes (\delta_{\bfw} - \delta_{\what}) \bfn , \bfv \rangle_{\partial \calT_h} \\
&+ \langle (\tau_C(\what_1) - \tau_C(\what_2))(\bfu - \uhat) , \bfv - \vhat \rangle_{\partial \calT_h} + \langle (\uhat \otimes \delta_{\what}) \bfn , \bfv - \vhat \rangle_{\partial \calT_h}, \\
\intertext{applying intergration by parts in the first term, rearranging the terms, we have}
& = (\frac12 (\nabla \cdot \delta_{\bfw}) \bfu, \bfv)_{\calT_h} + (\bfv \otimes \delta_{\bfw}, \nabla \bfu)_{\calT_h} +  \langle \frac12 \bfu \otimes (\delta_{\bfw} - \delta_{\what}) \bfn , \bfv \rangle_{\partial \calT_h} \\ 
&+ \langle (\tau_C(\what_1) - \tau_C(\what_2))(\bfu - \uhat) , \bfv - \vhat \rangle_{\partial \calT_h} + \langle (\uhat \otimes \delta_{\what} - \bfu \otimes \delta_{\bfw}) \bfn , \bfv \rangle_{\partial \calT_h} \\
& = T_1 + T_2 + T_3 + T_4 + T_5.
\end{align*}
Next we estimate each $T_i$. 

For $T_1$, we apply the generalized H\"{o}lder's inequality, 
\[
T_1 \le \|\nabla \cdot \delta_{\bfw}\|_{\calT_h} \|\bfu\|_{L^4(\Omega)} \|\bfv\|_{L^4(\Omega)} \le C \nrm (\delta_{\bfw}, \delta_{\what}) \nrm_{1, h} \; \nrm (\bfu, \uhat) \nrm_{1, h} \;  \nrm (\bfv, \vhat) \nrm_{1, h},
\]
the second inequality is due to \eqref{basic_2}. $T_2$ can be bounded in a similar way. 

For $T_3$, we apply the generalized H\"{o}lder's inequality,
\[
T_3 \le h^{-\frac12} \|\delta_{\bfw} - \delta_{\what} \|_{L^2(\partial \calT_h)} \; h^{\frac14}\| \bfu \|_{L^4(\partial \calT_h)} \; h^{\frac14}\| \bfv \|_{L^4(\partial \calT_h)} \le C  \nrm (\delta_{\bfw}, \delta_{\what}) \nrm_{1, h} \; \nrm (\bfu, \uhat) \nrm_{1, h} \;  \nrm (\bfv, \vhat) \nrm_{1, h},
\]
the second inequality is by \eqref{basic_3}.

For $T_4$, by the generalized H\"{o}lder's inequality we have:
\begin{align*}
T_4 &\le C \| \tau_C(\what_1) - \tau_C(\what_2) \|_{L^4(\partial \calT_h)} \;  \| \bfu - \uhat \|_{L^2(\partial \calT_h)}  \; \| \bfv - \vhat \|_{L^4(\partial \calT_h)}, \\
\intertext{by the fact that the function $\max(\bfw \cdot \bfn, 0)$ is Lipschitz,}
& \le C \| \delta_{\what} \|_{L^4(\partial \calT_h)} \;  \| \bfu - \uhat \|_{L^2(\partial \calT_h)}  \; \| \bfv - \vhat \|_{L^4(\partial \calT_h)} \\
 & \le C h^{\frac12}  ( \| \delta_{\bfw} - \delta_{\what} \|_{L^4(\partial \calT_h)} + \| \delta_{\bfw} \|_{L^4(\partial \calT_h)}) \;  
 \nrm ({\bfu}, {\uhat}) \nrm_{1, h} \;  \| \bfv - \vhat \|_{L^4(\partial \calT_h)}, \\
\intertext{Notice here if $(\delta_{\bfw}, \delta_{\what}) \in \widetilde{\bfH_0^1}(\Omega) +\big(\bfV_h \times \bfM^0_h \big)$, 
then $ \delta_{\bfw} - \delta_{\what} \in \bfV_h|_{\calT_h}$. Hence, we can apply inverse inequality on $ \| \delta_{\bfw} 
- \delta_{\what} \|_{L^4(\partial \calT_h)},   \| \bfv - \vhat \|_{L^4(\partial \calT_h)}$ to have}
& \le C h^{\frac12}  ( h^{\frac{1-d}{4}}\| \delta_{\bfw} - \delta_{\what} \|_{L^2(\partial \calT_h)} 
+ \| \delta_{\bfw} \|_{L^4(\partial \calT_h)}) \;  \nrm ({\bfu}, {\uhat}) \nrm_{1, h} \;  h^{\frac{1-d}{4}}  \| \bfv - \vhat \|_{L^2(\partial \calT_h)} \\
& \le C h^{\frac14} \nrm (\delta_{\bfw}, \delta_{\what}) \nrm_{1, h} \; \nrm (\bfu, \uhat) \nrm_{1, h} \;  \nrm (\bfv, \vhat) \nrm_{1, h},
\end{align*}
by \eqref{basic_3} and the fact that $d = 2, 3$. 

Finally, for $T_5$ we first break it into two terms:
\[
T_5 =  - \langle (\uhat \otimes (\delta_{\bfw} - \delta_{\what})) \bfn , \bfv \rangle_{\partial \calT_h} 
-  \langle ((\bfu - \uhat) \otimes \delta_{\bfw}) \bfn , \bfv \rangle_{\partial \calT_h}.
\]
For the first term, by the generalized H\"{o}lder's inequality, we have:
\begin{align*}
\langle (\uhat \otimes (\delta_{\bfw} - \delta_{\what})) \bfn , \bfv \rangle_{\partial \calT_h} &\le h^{\frac14}\| \uhat \|_{L^4(\partial \calT_h)} \;  
h^{-\frac12}\| \delta_{\bfw} - \delta_{\what} \|_{L^2(\partial \calT_h)}  \; h^{\frac14} \| \bfv \|_{L^4(\partial \calT_h)} \\
&\le h^{\frac14}(\| \bfu - \uhat \|_{L^4(\partial \calT_h)} + \| \bfu\|_{L^4(\partial \calT_h)})  \;  h^{-\frac12}\| \delta_{\bfw} 
- \delta_{\what} \|_{L^2(\partial \calT_h)}  \; h^{\frac14} \| \bfv \|_{L^4(\partial \calT_h)} \\
& \le C \nrm (\delta_{\bfw}, \delta_{\what}) \nrm_{1, h} \; \nrm (\bfu, \uhat) \nrm_{1, h} \;  \nrm (\bfv, \vhat) \nrm_{1, h}.
\end{align*}
In the last step we used the same argument as in $T_4$ for the term $\| \bfu - \uhat \|_{L^4(\partial \calT_h)}$ and \eqref{basic_3}. 
The second term can be estimated in a similar way. We complete the proof by combining the estimates of $T_i$.
\end{proof}

\section{Uniqueness and existence of the numerical solution.}
We will apply the Picard fixed point theorem to show the existence and uniqueness of the solution of \eqref{eq:NSEHDG}. 
To this end, we begin by rewriting the method into a more compact and appropriate form for the proof. 
If we add \eqref{eq:NSEHDGI} - \eqref{eq:NSEHDGIV}, the method can be written as: Find $(\Lm_h, \bfu_h, p_h,\uhat_h)\in 
\Gm_h\times \bfV_h\times Q_h\times\bfM_h^{\,0}$ such that
\begin{equation}\label{compact_form}
\mathcal{S}((\Lm_h, \bfu_h, p_h,\uhat_h), ((\Gm,\bfv, q,\bfmu))) + \calO((\bfu_h, \uhat_h); (\bfu_h, \uhat_h), (\bfv, \bfmu)) = (\bff, \bfv)_{\calT_h}, 
\end{equation}
for all $(\Gm,\bfv, q,\bfmu) \in \Gm_h\times \bfV_h\times Q_h\times
\bfM_h^{\,0}$. Here the bilinear form $\mathcal{S}(\cdot, \cdot)$ is defined as:
\begin{align*}
\mathcal{S}((\Lm_h, \bfu_h, p_h,\uhat_h), (\Gm,\bfv, q,\bfmu)) &:=  (\Lm_h, \Gm)_{\calT_h}+(\bfu_h, \nabla \cdot \Gm)_{\calT_h}-\langle \uhat_h, \Gm\bfn\rangle_{\partial \calT_h} - (\bfv, \nabla \cdot \nu \Lm_h)_{\calT_h} + \langle \bfmu, \nu \Lm_h \bfn \rangle_{\partial \calT_h} \\
& \;\; -(\bfu_h, \nabla q)_{\calT_h}+\langle\uhat_h\cdot \bfn,q\rangle_{\partial\calT_h} + (\bfv, \nabla p_h)_{\calT_h} - \langle \bfmu \cdot \bfn,p_h \rangle_{\partial\calT_h}\\
&\;\; + \langle \frac{\nu}{h} (\itPi_M \bfu_h - \uhat_h), \bfv - \bfmu \rangle_{\partial \calT_h}.
\end{align*}

We also define a mapping $\mathcal{F}$ as follows: for any $(\bfw, \what) \in \bfH^1(\calT_h) \times \bfL^2(\calE_h)$, $(\bfu_h, \uhat_h) 
= \mathcal{F}(\bfw, \what) \in \bfV_h \times \bfM_h$ is part of the solution $(\Lm_h, \bfu_h, p_h, \uhat_h)\in \Gm_h\times \bfV_h\times Q_h
\times \bfM_h^{\,0}$ of
\begin{equation}\label{eq:Oseen}
\mathcal{S}((\Lm_h, \bfu_h, p_h, \uhat_h), (\Gm,\bfv, q,\bfmu)) + \calO((\bfw, \what); (\bfu_h, \uhat_h), (\bfv, \bfmu)) = (\bff, \bfv)_{\calT_h},
\end{equation}
for all $(\Gm,\bfv, q,\bfmu) \in \Gm_h\times \bfV_h\times Q_h\times \bfM_h^{\,0}$. It is worth to mention that when $\bfw \in \bfH(\text{div}, \Omega), 
\nabla \cdot \bfw = 0$ and $\what = \bfw|_{\calE_h}$, the above system is a HDG scheme for the Oseen equation. Clearly, $(\bfu_h, \uhat_h)$ is 
a solution of \eqref{eq:NSEHDG} if and only if it is a fixed point of the mapping $\mathcal{F}$.  Next we present a stability result for the above scheme. 

\begin{lemma}\label{stability}
If $(\Lm_h, \bfu_h, p_h, \uhat_h)\in \Gm_h\times \bfV_h\times Q_h\times \bfM_h^{\,0}$ is a solution of \eqref{eq:Oseen}, 
then there exists a constant $C$ that solely depends on the constants $C_{\text{HDG}}$ and $C_2$ in Lemma \ref{H1_ineq}, \ref{Lq_ineq} such that
\[
\nrm (\bfu_h, \uhat_h) \nrm_{1,h} \le C \nu^{-1} \|\bff\|_{\Omega}.
\]
\end{lemma}
\begin{proof}
Taking $(\Gm,\bfv, q,\bfmu) = (\nu \Lm_h, \bfu_h, p_h, \uhat_h)$ in \eqref{eq:Oseen}, after some algebraic manipulation, we have a simplified equation:
\begin{align*}
\nu \|\Lm_h\|^2_{\Omega} + \langle \frac{\nu}{h} (\itPi_M \bfu_h - \uhat_h), \bfu_h - \uhat_h \rangle_{\partial \calT_h} 
+ \calO((\bfw, \what); (\bfu_h, \uhat_h), (\bfu_h, \uhat_h)) &= (\bff, \bfu_h)_{\Omega}.
\end{align*}
Therefore, by Lemma \ref{H1_ineq}, Young's inequality and Proposition \ref{Ocoercive} we have
\begin{align*}
\nu \nrm (\bfu_h, \uhat_h) \nrm^2_{1,h} &\le 2 C^2_{\text{HDG}} \nu (\|\Lm_h\|^2_{\Omega} + \frac{1}{h} \|\itPi_M \bfu_h - \uhat_h\|^2_{\partial \calT_h}) \\
& \le 2  C^2_{\text{HDG}} \left( \nu \| \Lm_h\|^2_{\Omega} + \langle \frac{\nu}{h} (\itPi_M \bfu_h - \uhat_h), \bfu_h - \uhat_h \rangle_{\partial \calT_h} 
+ \calO((\bfw, \what); (\bfu_h, \uhat_h), (\bfu_h, \uhat_h)) \right) \\
& = 2 C^2_{\text{HDG}} (\bff, \bfu_h)_{\Omega} \le 2 C^2_{\text{HDG}} \|\bfu_h\|_{\Omega} \|\bff\|_{\Omega} \le 2 C_2 C^2_{\text{HDG}} 
\nrm (\bfu_h, \uhat_h) \nrm_{1,h} \|\bff\|_{\Omega}.
\end{align*}
The last step is by Lemma \ref{Lq_ineq} with $q = 2$. This completes the proof with $C = 2 C_2 C^2_{\text{HDG}}$. 
\end{proof}

Inspired by the above stability result, we define a subspace of $\bfV_h \times \bfM_h^0$:
\[
\mathcal{K}_h := \{(\bfv,\bfmu)\in \bfV_h \times \bfM_h^0:\nrm(\bfv,\bfmu)\nrm_{1,h}\leq C_2C_{\rm HDG}^2\nu^{-1}\|\bff\|_{\Omega}\}.
\]

We are now ready to give the proof of the existence and uniqueness result for the HDG scheme \eqref{eq:NSEHDG}/\eqref{compact_form}.

\begin{proof} {\bf of Theorem \ref{thm:existence}.} \

Clearly, $\mathcal{F}$ maps $\bfV_h \times \bfM_h^0$ into $\mathcal{K}_h$ hence it maps  $\mathcal{K}_h$ into itself. 
In order to show the existence and uniqueness of the solution of \eqref{eq:NSEHDG}/\eqref{compact_form}, it suffices to show that 
$\mathcal{F}$ is a contraction on the subspace $\mathcal{K}_h$. To this end, let $(\bfw_1,\what_1), (\bfw_2,\what_2) \in K_h$ and 
$(\Lm_i, \bfu_i, p_i, \uhat_i)$ are the solutions of the problem \eqref{eq:Oseen} with $(\bfw,\what) = (\bfw_i, \what_i)$, $(i = 1, 2)$. So we have
 $(\bfu_1,\uhat_1):=\calF(\bfw_1,\what_1)$ and $(\bfu_2,\uhat_2):=\calF(\bfw_2,\what_2)$. 
If we set $\delta_{\Lm}:=L_1-L_2$, $\delta_{\bfu}:=\bfu_1-\bfu_2$,
$\delta_p:=p_1-p_2$ and $\delta_{\uhat}:=\uhat_1-\uhat_2$, due to the linearity of the operator $\mathcal{S}$, we have
\[
\mathcal{S}( (\delta_{\Lm}, \delta_{\bfu}, \delta_p, \delta_{\uhat}), (\Gm,\bfv, q,\bfmu)) 
+ \calO((\bfw_1, \what_1); (\bfu_1, \uhat_1), (\bfv, \bfmu)) - \calO((\bfw_2, \what_2); (\bfu_2, \uhat_2), (\bfv, \bfmu)) = 0,
\]
for all $(\Gm,\bfv, q,\bfmu) \in \Gm_h\times \bfV_h\times Q_h\times \bfM_h^{\,0}$. Taking $(\Gm,\bfv, q,\bfmu) 
= (\nu \delta_{\Lm}, \delta_{\bfu}, \delta_p, \delta_{\uhat})$ into the above identity, after some algebraic manipulations, we obtain
\[
\nu \|\delta_{\Lm}\|^2_{\Omega} + \langle \frac{\nu}{h} (\itPi_M \delta_{\bfu} - \delta_{\uhat}),  
\delta_{\bfu} - \delta_{\uhat} \rangle_{\partial \calT_h} =  - \calO((\bfw_1, \what_1); (\bfu_1, \uhat_1), 
(\delta_{\bfu}, \delta_{\uhat})) + \calO((\bfw_2, \what_2); (\bfu_2, \uhat_2), (\delta_{\bfu}, \delta_{\uhat})).
\]
Or
\begin{align*}
\nu \|\delta_{\Lm}\|^2_{\Omega} + \langle \frac{\nu}{h} (\itPi_M \delta_{\bfu} - \delta_{\uhat}),  \delta_{\bfu} - \delta_{\uhat} \rangle_{\partial \calT_h} + & \calO((\bfw_2, \what_2); (\delta_{\bfu}, \delta_{\uhat}), (\delta_{\bfu}, \delta_{\uhat})) = \\
& \calO((\bfw_2, \what_2); (\bfu_1, \uhat_1), (\delta_{\bfu}, \delta_{\uhat}))  -  \calO((\bfw_1, \what_1); (\bfu_1, \uhat_1), (\delta_{\bfu}, \delta_{\uhat})).
\end{align*}

By Lemma \ref{H1_ineq} Young's inequality and Proposition \ref{Ocoercive} we have
\begin{align*}
\nu \nrm (\delta_{\bfu}, \delta_{\uhat}) \nrm^2_{1,h} &\le 2 C^2_{\text{HDG}} (\nu \|\delta_{\Lm}\|^2_{\Omega} 
+ \frac{\nu}{h} \|\itPi_M \delta_{\bfu} - \delta_{\uhat}\|^2_{\partial \calT_h}) \\
& \le 2 C^2_{\text{HDG}} \left(\nu \|\delta_{\Lm}\|^2_{\Omega} + \langle \frac{\nu}{h} (\itPi_M \delta_{\bfu} - \delta_{\uhat}),  
\delta_{\bfu} - \delta_{\uhat} \rangle_{\partial \calT_h} +  \calO((\bfw_2, \what_2); (\delta_{\bfu}, \delta_{\uhat}), (\delta_{\bfu}, \delta_{\uhat}))\right)\\
& =  2 C^2_{\text{HDG}} \; \Big( \calO((\bfw_2, \what_2); (\bfu_1, \uhat_1), (\delta_{\bfu}, \delta_{\uhat}))  
-  \calO((\bfw_1, \what_1); (\bfu_1, \uhat_1), (\delta_{\bfu}, \delta_{\uhat})) \Big) \\
& \le 2 C_{\calO}  C^2_{\text{HDG}} \nrm (\bfw_1 - \bfw_2, \what_1 - \what_2 \nrm_{1,h} \;  \nrm (\bfu_1, \uhat_1) \nrm_{1,h} \; 
\nrm (\delta_{\bfu}, \delta_{\uhat}) \nrm_{1,h} \quad \text{by Lemma \ref{O-estimates},} \\
& \le 2 C_{\calO} C_2 C^4_{\text{HDG}} \nu^{-1} \|\bff\|_{\Omega} \; \nrm (\bfw_1 - \bfw_2, \what_1 - \what_2 \nrm_{1,h} \;  
\nrm (\delta_{\bfu}, \delta_{\uhat}) \nrm_{1,h} \quad \;\; \, \text{by Lemma \ref{stability}}.
\end{align*}
Therefore, we have shown that
\[
\nrm (\delta_{\bfu}, \delta_{\uhat}) \nrm_{1,h} \le 2 C_{\calO} C_2 C^4_{\text{HDG}} \nu^{-2} \|\bff\|_{\Omega} \; 
\nrm (\bfw_1 - \bfw_2, \what_1 - \what_2 \nrm_{1,h}.
\]
Obviously, the above bound implies that $\mathcal{F}$ is a contraction on $\mathcal{K}_h$ equipped with $\nrm \cdot \nrm_{1,h}$ provided
\[
\|\bff\|_{\Omega} \le \frac{\nu^2}{ 2 C_{\calO} C_2 C^4_{\text{HDG}}}.
\]
By the fixed point theorem, there is a unique fixed point $(\bfu_h, \uhat_h) \in \mathcal{K}_h$ of the mapping $\mathcal{F}$. 
It is also the unique solution of the system \eqref{eq:NSEHDG}. This completes the proof.
\end{proof}

\section{Proof of the error estimates}
In this section, we provide the detailed proof of the main error estimates for all unknowns. We proceed in several steps.

{\bf Step 1: Error equations.} We begin by introducing the error equations that we are going to use in the analysis. For convention, we introduce the following notations for the errors:
\begin{align*}
e_{\Lm} &:= \itPi_G \Lm - \Lm_h, \quad  && e_{\bfu} := \itPi_V \bfu - \bfu_h, \quad && e_p := \itPi_Q p - p_h, \quad && e_{\uhat} := \itPi_M \bfu - \bfu_h, \\
\delta_{\Lm} &:= \Lm - \itPi_G \Lm, \quad && \delta_{\bfu} := \bfu - \itPi_V \bfu, \quad && \delta_p := p - \itPi_Q p, \quad  && \delta_{\uhat} := \bfu - \itPi_M \bfu.
\end{align*} 
Here $\itPi_G, \itPi_V, \itPi_Q, \itPi_M$ are the $L^2-$projections onto $\Gm_h, \bfV_h, Q_h, \bfM_h$ respectively. In addition to the basic inequalities listed in Lemma \ref{Lq_ineq}, we will frequently use the following basic inequalities as well:
\begin{subequations}
\begin{align}
\label{basic_5}
\|q\|_{F} &\le C h^{-\frac12}_K \|q\|_{K}, && \text{for all $q \in P_l(K), (l \ge 0)$,}\\
\label{basic_6}
\|D^{m} (q - \itPi_l q) \|_{K} &\le C h^{l+1 - m}_K \|q\|_{l+1, K},  && \text{for all $q \in H^{l+1}(K)$, $0 \le m \le l$,} \\
\label{basic_6a}
\|q - \itPi_l q \|_{F} &\le C h^{l+\frac12}_K \|q\|_{l+1, K},  && \text{for all $q \in H^{l+1}(K)$, $0 \le m \le l$,} \\
\label{basic_PM}
\|q - \itPi_M q\|_{F} & \le C h^{k+\frac12} \|q\|_{k+1, K} && \text{for all $q \in H^{l+1}(K)$,}
\end{align}
\end{subequations}
Here $\itPi_l$ denotes the $L^2-$projection onto $P_l(K)$, $F$ denotes any face of $K$. In addition, we have the following estimate for the projections under the triple norm $\nrm \cdot \nrm_{1,h}$:
\begin{proposition}\label{basic_7}
For any $\bfu \in \bfH^1(\Omega)$, we have
\begin{equation}\label{eq:triple_norm_est}
\nrm (\itPi_V \bfu, \itPi_M \bfu ) \nrm_{1, h} \le C \|\bfu\|_{1, \Omega}.
\end{equation}
\end{proposition}

\begin{proof}
By the definition of the norm $\nrm \cdot \nrm_{1,h}$, we have
\[
\nrm (\itPi_V \bfu, \itPi_M \bfu ) \nrm_{1, h} \le 2( \|\nabla \itPi_V \bfu\|_{\calT_h} + h^{-\frac12} \|\itPi_V \bfu - \itPi_M \bfu \|_{\partial \calT_h}).
\]
We are going to bound each of the above terms by $\|\bfu\|_{1, \Omega}$. For the first term, we have
\begin{align*}
\|\nabla \itPi_V \bfu\|_{\calT_h} &= \|\nabla ( \itPi_V \bfu - \bar{\bfu})\|_{\calT_h} \le C h^{-1} \|\itPi_V \bfu - \bar{\bfu}\|_{0, \calT_h}, \\
\intertext{here $\bar{\bfu}$ denotes the average of $\bfu$ within each element $K \in \calT_h$, the inequality is by the inverse inequality of the polynomial spaces.} 
\|\nabla \itPi_V \bfu\|_{\calT_h}  & \le C \|\bfu\|_{1, h} = C \|\bfu\|_{1, \Omega}, 
\end{align*}
by the Poincar\'{e} inequality for each $K \in \calT_h$.

For the second term, applying a triangle inequality we have
\begin{align*}
h^{-\frac12} \|\itPi_V \bfu - \itPi_M \bfu \|_{\partial \calT_h} & \le h^{-\frac12} ( \| \bfu - \itPi_V \bfu \|_{\partial \calT_h} + \| \bfu - \itPi_M \bfu \|_{\partial \calT_h}) \le 2 h^{-\frac12} \| \bfu - \itPi_k \bfu \|_{\partial \calT_h},  \\
\intertext{here $\itPi_k$ denotes the $L^2-$projection onto $\bfP_k(K)$ for each $K \in \calT_h$,}
& \le C ( \|\nabla (\bfu - \itPi_k \bfu)\|_{\calT_h} + h^{-1} \| \bfu - \itPi_k \bfu\|_{\calT_h}) \quad \text{by the trace inequality,} \\
&\le C \|\bfu\|_{1, \Omega},
\end{align*}
the last step is by a similar argument as for the first term and \eqref{basic_7}. This completes the proof.
\end{proof}

It is not hard to verify that the exact solution $(\bfu, \Lm, p, \bfu|_{\calE_h})$ satisfies the following equation:
\[
\mathcal{S}((\Lm, \bfu, p,\uhat), ((\Gm,\bfv, q,\bfmu))) + \calO((\bfu, \uhat); (\bfu, \uhat), (\bfv, \bfmu)) = (\bff, \bfv)_{\calT_h} + \langle \frac{\nu}{h} \delta_{\uhat} \, , \, \bfv - \mu \rangle_{\partial \calT_h} , 
\]
for all $(\Gm,\bfv, q,\bfmu) \in \Gm_h\times \bfV_h\times Q_h\times
\bfM_h^{\,0}$. Subtracting \eqref{compact_form}, we have
\begin{align*}
\mathcal{S}((\Lm, \bfu, p, \uhat), ((\Gm,\bfv, q,\bfmu))) &- \mathcal{S}((\Lm_h, \bfu_h, p_h, \uhat_h), ((\Gm,\bfv, q,\bfmu))) \\
& + \calO((\bfu, \bfu); (\bfu, \bfu), (\bfv, \bfmu)) - \calO((\bfu_h, \uhat_h); (\bfu_h, \uhat_h), (\bfv, \bfmu)) = \langle \frac{\nu}{h} \delta_{\uhat} \, , \, \bfv - \mu \rangle_{\partial \calT_h}, \\
\intertext{by the linearility of the first operator $\mathcal{S}$, we have}
\mathcal{S}((e_{\Lm}, e_{\bfu}, e_p, e_{\uhat}), ((\Gm,\bfv, q,\bfmu))) &+ \calO((\bfu, \bfu); (\bfu, \bfu), (\bfv, \bfmu)) - \calO((\bfu_h, \uhat_h); (\bfu_h, \uhat_h), (\bfv, \bfmu)) = \\
& - \mathcal{S}((\delta_{\Lm}, \delta_{\bfu}, \delta_p, \delta_{\uhat}), ((\Gm,\bfv, q,\bfmu))) + \langle \frac{\nu}{h} \delta_{\uhat} \, , \, \bfv - \mu \rangle_{\partial \calT_h}.
\end{align*}
Finally, by the definition of the operator $\mathcal{S}$ and the orthogonality property of the $L^2-$projections we have the error equation:
\begin{equation}\label{eq:error}
\begin{aligned}
\mathcal{S}((e_{\Lm}, e_{\bfu}, e_p, e_{\uhat}), ((\Gm,\bfv, q,\bfmu))) + \calO((\bfu, \bfu); (\bfu, \bfu), (\bfv, \bfmu)) - & \calO((\bfu_h, \uhat_h); (\bfu_h, \uhat_h), (\bfv, \bfmu)) = \\
& \langle \nu \delta_{\Lm} \bfn - \delta_p \bfn - \frac{\nu}{h} \itPi_M \delta_{\bfu} \, , \, \bfv - \mu \rangle_{\partial \calT_h},
\end{aligned}
\end{equation}
for all $(\Gm,\bfv, q,\bfmu) \in \Gm_h\times \bfV_h\times Q_h\times \bfM_h^{\,0}$.

{\bf Step 2: Estimates for $e_{\Lm}$, $e_{\bfu}$.} We first apply an energy argument to bound the errors $e_{\Lm}$, $e_{\bfu}$, which can be stated as follows:
\begin{lemma}\label{estimate_EL}
If the exact solution $\bfu, \Lm, p$ is smooth enough, and $\|\bfu\|_{1, \Omega}$ is small enough, we have
\[
\|e_{\bfu}\|_{\Omega} + \nrm (e_{\bfu}, e_{\uhat}) \nrm_{1, h} \le C_{\text{HDG}} ( \|e_{\Lm}\|_{\Omega}  + h^{-\frac12} \|\itPi_M e_{\bfu} - e_{\uhat}\|_{\partial \calT_h} )\le \mathcal{C} h^{k+1}.
\]
Here the constant $\mathcal{C}$ depends on $\|\bfu\|_{k+2, \Omega}, \|\bfu\|_{\infty, \Omega}, \|p\|_{k+1, \Omega}, \nu$ and $k$ but independent of $h$. 
\end{lemma}

\begin{proof}
Taking $(\Gm,\bfv, q,\bfmu) = (e_{\Lm}, e_{\bfu}, e_p, e_{\uhat})$ in the error equation \eqref{eq:error}, the resulting equation can be simplified as

\begin{align*}
\nu \|e_{\Lm}\|^2_{\Omega} + \frac{\nu}{h} \|\itPi_M e_{\bfu} - e_{\uhat}\|^2_{\partial \calT_h} + \calO((\bfu, \bfu); (\bfu, \bfu), (e_{\bfu}, e_{\uhat})) - & \calO((\bfu_h, \uhat_h); (\bfu_h, \uhat_h),  (e_{\bfu}, e_{\uhat})) = \\
& \langle \nu \delta_{\Lm} \bfn - \delta_p \bfn - \frac{\nu}{h} \itPi_M \delta_{\bfu} \, , \, e_{\bfu} - e_{\uhat} \rangle_{\partial \calT_h},
\end{align*}
or
\begin{equation}\label{eq:energy}
\begin{aligned}
\nu \|e_{\Lm}\|^2_{\Omega} + \frac{\nu}{h} \|\itPi_M e_{\bfu} - e_{\uhat}\|^2_{\partial \calT_h} =  & \calO((\bfu_h, \uhat_h); (\bfu_h, \uhat_h), (e_{\bfu}, e_{\uhat})) - \calO((\bfu, \bfu); (\bfu, \bfu), (e_{\bfu}, e_{\uhat})) \\
& + \langle \nu \delta_{\Lm} \bfn - \delta_p \bfn - \frac{\nu}{h} \itPi_M \delta_{\bfu} \, , \, e_{\bfu} - e_{\uhat} \rangle_{\partial \calT_h}.
\end{aligned}
\end{equation}
Let us first estimate the last term of the above equation. Applying the Cauchy-Schwarz inequality, we have
\begin{align*}
\langle \nu \delta_{\Lm} \bfn - \delta_p \bfn - \frac{\nu}{h} \itPi_M \delta_{\bfu} \, , \, e_{\bfu} - e_{\uhat} \rangle_{\partial \calT_h} & \le h^{\frac12} \Big(\nu \|\delta_{\Lm}\|_{\partial \calT_h} + \|\delta_p\|_{\partial \calT_h} + \nu h^{-1} \| \itPi_M \delta_{\bfu}\|_{\partial \calT_h}\Big) h^{-\frac12} \|e_{\bfu} - e_{\uhat}\|_{\partial \calT_h} \\
& \le h^{\frac12} \Big(\nu \|\delta_{\Lm}\|_{\partial \calT_h} + \|\delta_p\|_{\partial \calT_h} + \nu h^{-1} \| \delta_{\bfu}\|_{\partial \calT_h}\Big) h^{-\frac12} \|e_{\bfu} - e_{\uhat}\|_{\partial \calT_h} \\
& \le C \Big( \|\delta_{\Lm}\|_{\Omega} + \|\delta_p\|_{\Omega} + h^{-1} \| \delta_{\bfu}\|_{\Omega}\Big) \nrm (e_{\bfu}, e_{\uhat})\nrm_{1, h} \quad \text{by \eqref{basic_5}}, \\
& \le C h^{k+1} (\|\Lm\|_{k+1, \Omega} + \|p\|_{k+1, \Omega} + \|\bfu\|_{k+2, \Omega}) \nrm (e_{\bfu}, e_{\uhat})\nrm_{1, h} \quad \text{by \eqref{basic_6}}.
\end{align*}
The main effort in the analysis is to have an optimal estimate for the nonlinear terms, we rewrite these two terms into four parts:
\begin{align*}
& \calO((\bfu_h, \uhat_h); (\bfu_h, \uhat_h), (e_{\bfu}, e_{\uhat})) - \calO((\bfu, \bfu); (\bfu, \bfu), (e_{\bfu}, e_{\uhat})) \quad && \\
= \; &  \calO((\bfu_h, \uhat_h); (\bfu_h, \uhat_h), (e_{\bfu}, e_{\uhat})) - \calO((\bfu_h, \uhat_h); (\itPi_V \bfu,\itPi_M \bfu), (e_{\bfu}, e_{\uhat})) && : T_1 \quad && \\
+ \; & \calO((\bfu_h, \uhat_h); (\itPi_V \bfu,\itPi_M \bfu), (e_{\bfu}, e_{\uhat})) - \calO((\itPi_V \bfu,\itPi_M \bfu); (\itPi_V \bfu,\itPi_M \bfu), (e_{\bfu}, e_{\uhat})) && :T_2 \\
+ \; & \calO((\itPi_V \bfu,\itPi_M \bfu); (\itPi_V \bfu,\itPi_M \bfu), (e_{\bfu}, e_{\uhat})) - \calO(( \bfu, \bfu); (\itPi_V \bfu,\itPi_M \bfu), (e_{\bfu}, e_{\uhat})) && : T_3 \\
+ \; & \calO((\bfu, \bfu); (\itPi_V \bfu,\itPi_M \bfu), (e_{\bfu}, e_{\uhat})) - \calO(( \bfu, \bfu); ( \bfu, \bfu), (e_{\bfu}, e_{\uhat})) && : T_4
\end{align*}
Notice that the operator $\mathcal{O}$ is linear with respect to the last two components, so we have
\[
T_1 = - \calO((\bfu_h, \uhat_h); (e_{\bfu}, e_{\uhat}), (e_{\bfu}, e_{\uhat})) \le 0, 
\]
by Proposition \ref{Ocoercive}. For $T_2$, we apply the estimate \eqref{O_estimate1} in Lemma \ref{O-estimates}, we have
\[
T_2 \le C_{\mathcal{O}} \nrm (\itPi_V \bfu,\itPi_M \bfu) \nrm_{1,h} \nrm (e_{\bfu}, e_{\uhat}) \nrm^2_{1,h} \le C_{\mathcal{O}} \|\bfu\|_{1, \Omega} \nrm (e_{\bfu}, e_{\uhat}) \nrm^2_{1,h}. 
\]

For $T_3, T_4$, if we directly apply the continuity property of the operator $\mathcal{O}$ \eqref{O_estimate1}, we will only obtain suboptimal order of convergence. We can recover optimal convergence rate by a refined argument for each term. We start with $T_4$, by the linearity of $\mathcal{O}$ for the last two components, we have
\begin{align*}
T_4 & = - \calO((\bfu, \bfu); (\delta_{\bfu}, \delta_{\uhat}), (e_{\bfu}, e_{\uhat})) \\
& = ( \delta_{\bfu} \otimes \bfu , \nabla e_{\bfu})_{\calT_h} - \langle \tau_C(\bfu) (\delta_{\bfu} - \delta_{\uhat}), e_{\bfu} - e_{\uhat} \rangle_{\partial \calT_h} - \langle ( \delta_{\uhat} \otimes \bfu ) \bfn, e_{\bfu} - e_{\uhat} \rangle_{\partial \calT_h} := T_{41} + T_{42} + T_{43}.
\end{align*}
We now apply the generalized H\"{o}lder's inequality to bound these three terms as follows:
\begin{align*}
T_{41} &\le  \sum_{K \in \calT_h} \|\bfu\|_{\infty, K} \|\delta_{\bfu}\|_{K} \|\nabla e_{\bfu}\|_{K} 
\le C h^{k+2}\|\bfu\|_{\infty, \Omega} \|\bfu\|_{k+2, \Omega} \nrm (e_{\bfu}, e_{\uhat}) \nrm_{1, h}, \\
T_{42} & \le \sum_{K \in \calT_h} \|\bfu \cdot \bfu\|_{\infty, \partial K} 
h^{\frac12} \|\delta_{\bfu} - \delta_{\uhat}\|_{\partial K} h^{-\frac12} \|e_{\bfu} - e_{\uhat}\|_{\partial K} 
\le C \|\bfu \cdot \bfn\|_{\infty, \calE_h} h^{k+1} \|\bfu\|_{h+1, \Omega} \nrm (e_{\bfu}, e_{\uhat}) \nrm_{1, h}, \\
T_{43} & \le C \|\bfu \cdot \bfn\|_{\infty, \calE_h} h^{k+1} \|\bfu\|_{h+1, \Omega} \nrm (e_{\bfu}, e_{\uhat}) \nrm_{1, h}.
\end{align*}

For $T_3$, by the definition of $\mathcal{O}$, we have
\begin{align*}
{T_3} &= (\itPi_V \bfu \otimes \delta_{\bfu}, \nabla e_{\bfu})_{\calT_h} + (\frac12 (\nabla \cdot \delta_{\bfu}) \itPi_V \bfu, e_{\bfu})_{\calT_h} 
- \langle \frac12 (\itPi_V \bfu \otimes (\delta_{\bfu} - \delta_{\uhat})) \bfn , e_{\bfu} - e_{\uhat} \rangle_{\partial \calT_h} \\
& + \langle ( \tau_C(\itPi_M \bfu) - \tau_C(\bfu)) (\itPi_V \bfu - \itPi_M \bfu), e_{\bfu} - e_{\uhat} \rangle_{\partial \calT_h} 
- \langle (\itPi_M \bfu \otimes \delta_{\uhat}) \bfn , e_{\bfu} - e_{\uhat} \rangle_{\partial \calT_h}. 
\end{align*}
Among the five terms in the above expression, the third term needs some special treatment in order to obtain optimal convergence rates. 
For the others, we can bound them in a similar way as for $T_{41}, T_{42}, T_{43}$. For the sake of simplicity, here we show how to bound 
the last term and then focus on the third term. By applying the generalized H\"{o}lder's inequality on the last term, we have
\begin{align*}
\langle (\itPi_M \bfu \otimes \delta_{\uhat}) \bfn , e_{\bfu} - e_{\uhat} \rangle_{\partial \calT_h} 
&\le \sum_{K \in \calT_h} \|\itPi_M \bfu\|_{\infty, \partial K} h^{\frac12}\|\delta_{\uhat}\|_{\partial K} h^{-\frac12}\|e_{\bfu} - e_{\uhat}\|_{\partial K} \\
&\le C h^{k+1} \|\bfu\|_{\infty, \calE_h} \|\bfu\|_{k+1, \Omega} \nrm (e_{\bfu}, e_{\uhat}) \nrm_{1,h}.
\end{align*}
In the last step we applied the inequality:
\[
\|\itPi_M \bfu\|_{\infty, \partial K} \le C \|\bfu\|_{\infty, \partial K}.
\]
This result can be obtained by a simple scaling argument. Finally, let us focus on the third term. We rewrite the term as follows,

\begin{align*}
\langle \frac12 (\itPi_V \bfu \otimes (\delta_{\bfu} - \delta_{\uhat})) \bfn , e_{\bfu} \rangle_{\partial \calT_h}  
& = \langle \frac12 (\itPi_V \bfu \otimes (\delta_{\bfu} - \delta_{\uhat})) \bfn , e_{\bfu} - e_{\uhat} \rangle_{\partial \calT_h}  
+ \langle \frac12 (\itPi_V \bfu \otimes \delta_{\bfu}) \bfn , e_{\uhat} \rangle_{\partial \calT_h} \\ 
& \; - \langle \frac12 (\itPi_V \bfu \otimes  \delta_{\uhat}) \bfn , e_{\uhat} \rangle_{\partial \calT_h}  \\
& = T_{31} + T_{32} + T_{33}.
\end{align*}
For $T_{31}$, we apply the generalized H\"{o}lder's inequality, 
\[
T_{31} \le \sum_{K \in \calT_h} \| \itPi_V \bfu \|_{\infty, \partial K} h^{\frac12} \|\delta_{\bfu} 
- \delta_{\uhat}\|_{\partial K} h^{-\frac12} \|e_{\bfu} - e_{\uhat}\|_{\partial K} 
\le C h^{k+1} \|\bfu\|_{\infty, \Omega} \|\bfu\|_{k+1, \Omega} \nrm (e_{\bfu}, e_{\uhat}) \nrm_{1, h}.
\]
For $T_{32}$, we have
\[
T_{32} \le \sum_{K \in \calT_h} \|\itPi_V \bfu\|_{\infty, \partial K} h^{-\frac12} \|\delta_{\bfu}\|_{\partial K} \|e_{\uhat}\|_{\partial K} 
\le C h^{k+1} \|\bfu\|_{\infty, \Omega} \|\bfu\|_{k+2, \Omega} (h^{\frac12} \|e_{\uhat}\|_{\partial \calT_h}).
\]
For $T_{33}$, inserting a zero term $ \langle \frac12 ( \bfu \otimes  \delta_{\uhat}) \bfn , e_{\uhat} \rangle_{\partial \calT_h}$ into $T_{33}$ we obtain:
\begin{align*}
T_{33} & = \langle \frac12 ( \delta_{\bfu}  \otimes  \delta_{\uhat}) \bfn , e_{\uhat} \rangle_{\partial \calT_h} 
\le \sum_{K \in \calT_h} \|\delta_{\uhat}\|_{\infty, \partial K} h^{-\frac12} \|\delta_{\bfu}\|_{\partial K} ( h^{\frac12}\|e_{\uhat}\|_{\partial K}) 
\le C h^{k+1} \|\bfu\|_{\infty, \Omega} \|\bfu\|_{k+2, \Omega} (h^{\frac12} \|e_{\uhat}\|_{\partial \calT_h}).
\end{align*}

The last step is to show that
\[
h^{\frac12} \|e_{\uhat}\|_{\partial \calT_h} \le C \nrm (e_{\bfu}, e_{\uhat}) \nrm_{1, h}.
\]
To this end, we apply a triangle inequality and (\ref{basic_3}), 
\[
h^{\frac12} \|e_{\uhat}\|_{\partial \calT_h} \le h^{\frac12} \|e_{\bfu}\|_{\partial \calT_h} + h^{\frac12} \|e_{\bfu} - e_{\uhat}\|_{\partial \calT_h} 
\le C (\|e_{\bfu}\|_{\Omega} + h  \nrm (e_{\bfu}, e_{\uhat}) \nrm_{1, h})\le C  \nrm (e_{\bfu}, e_{\uhat}) \nrm_{1, h}.
\]
This completes the estimate for $T_3$. Finally, if we combine the estimates for $T_1 - T_4$ and Lemma \ref{H1_ineq}, 
we obtain the result stated in the Lemma \ref{estimate_EL}.
\end{proof}

{\bf Step 3: Estimates for $e_p$.} Next we present the optimal error estimate for $e_p$. As usual, we bound the pressure error 
via a \emph{inf-sup} argument. It is well-known \cite{BrezziFortin91} that the following \emph{inf-sup} condition holds for a polygonal domain $\Omega$:
\begin{equation}\label{inf_sup}
\sup_{\bfw \in \bfH^1_0(\Omega) \backslash \{ 0 \}} \frac{(\nabla \cdot \bfw , q)_{\Omega}}{\|\bfw\|_{1, \Omega}} \ge \kappa \|q\|_{\Omega}, \quad \text{for all} \; q \in L^2_0(\Omega).
\end{equation}
Here $\kappa > 0$ is independent of $\bfw, p$. We can bound $e_p$ using the above result.

\begin{lemma}\label{estimate_ep}
Under the same assumption as in Lemma \ref{estimate_EL}, we have
\[
\|e_p\|_{\Omega} \le \mathcal{C} h^{k+1}.
\]
Here the constant $\mathcal{C}$ depends on $\|\bfu\|_{k+2, \Omega}, \|\bfu\|_{\infty, \Omega}, \|p\|_{k+1, \Omega}, \nu, k$ and $\kappa$ but independent of $h$. 
\end{lemma}

\begin{proof}
Since $e_p \in L^2_0(\Omega)$, by \eqref{inf_sup} we have
\begin{equation}\label{inf_sup_p}
\|e_p\|_{\Omega} \le \frac{1}{\kappa} \sup_{\bfw \in \bfH^1_0(\omega) \backslash \{ 0 \}} \frac{(\nabla \cdot \bfw , e_p)_{\Omega}}{\|\bfw\|_{1, \Omega}}.
\end{equation}
Now let us work on the numerator. Applying integration by parts and the orthogonality property of the projections $\itPi_V, \itPi_M$, we can rewrite it as follows:
\[
(\nabla \cdot \bfw , e_p)_{\Omega} = (e_p, \nabla \cdot \itPi_V \bfw)_{\calT_h} + \langle (\bfw - \itPi_V \bfw) \cdot \bfn, e_p \rangle_{\partial \calT_h} =  (e_p, \nabla \cdot \itPi_V \bfw)_{\calT_h} - \langle (\itPi_V \bfw - \itPi_M \bfw) \cdot \bfn, e_p \rangle_{\partial \calT_h}.
\]
If we take $(\Gm,\bfv, q,\bfmu) = (0, \itPi_V \bfw, 0, \itPi_M \bfw)$ in the error equation \eqref{eq:error}, we obtain:
\begin{align*}
(e_p, \nabla \cdot \itPi_V \bfw)_{\calT_h} &- \langle (\itPi_V \bfw - \itPi_M \bfw) \cdot \bfn, e_p \rangle_{\partial \calT_h} \\
&= (\nu e_{\Lm}, \nabla \itPi_V \bfw)_{\calT_h} - \langle \nu e_{\Lm}  \bfn - \frac{\nu}{h} (\itPi_M e_{\bfu} - e_{\uhat}), \itPi_V \bfw - \itPi_M \bfw \rangle_{\partial \calT_h} \\
&+ \langle \nu \delta_{\Lm} \bfn - \delta_p \bfn - \frac{\nu}{h} \itPi_M \delta_{\bfu}, \itPi_V \bfw - \itPi_M \bfw \rangle_{\partial \calT_h} \\
& + \Big( \calO((\bfu, \bfu); (\bfu, \bfu), ( \itPi_V {\bfw}, \itPi_M {\bfw})) -  \calO((\bfu_h, \uhat_h); (\bfu_h, \uhat_h),  ( \itPi_V {\bfw}, \itPi_M {\bfw}))  \Big) \\
&:= T_1 + T_2 + T_3 + T_4.
\end{align*}
Next we show that 
\[
T_1 + T_2 + T_3 + T_4 \le C h^{k+1} \|\bfw\|_{1, \Omega}.
\]
For $T_1$ we have
\[
T_1 \le \nu \|e_{\Lm}\|_{\Omega} \|\nabla \itPi_V \bfw\|_{\calT_h} \le C h^{k+1} \|\bfw\|_{1, \Omega},
\]
by Lemma \ref{estimate_EL} and \eqref{basic_7}. For $T_2$, by the Cauchy-Schwarz inequality, we have
\begin{align*}
T_2 &\le \nu (h^{\frac12}\|e_{\Lm}\|_{\partial \calT_h} + h^{-\frac12} \|\itPi_M e_{\bfu} - e_{\uhat}\|_{\partial \calT_h}) (h^{-\frac12} \|\itPi_V \bfw - \itPi_M \bfw\|_{\partial \calT_h}) \\
& \le C h^{k+1} \|\bfw\|_{1, \Omega}.
\end{align*}
$T_3$ can be estimated similarly as $T_2$. Finally, for the last term, we split into three terms as:
\begin{align*}
T_4 &= \calO((\bfu, \bfu); (\bfu, \bfu), ( \itPi_V {\bfw}, \itPi_M {\bfw})) -  \calO((\bfu, \bfu); (\itPi_V \bfu, \itPi_M \bfu),  ( \itPi_V {\bfw}, \itPi_M {\bfw})) \quad && :T_{41}\\
&+ \calO((\bfu, \bfu); (\itPi_V \bfu, \itPi_M \bfu),  ( \itPi_V {\bfw}, \itPi_M {\bfw})) -  \calO((\bfu_h, \uhat_h); (\itPi_V \bfu, \itPi_M \bfu),  ( \itPi_V {\bfw}, \itPi_M {\bfw}))  && :T_{42} \\
& + \calO((\bfu_h, \uhat_h); (\itPi_V \bfu, \itPi_M \bfu),  ( \itPi_V {\bfw}, \itPi_M {\bfw})) - \calO((\bfu_h, \uhat_h); (\bfu_h, \uhat_h),  ( \itPi_V {\bfw}, \itPi_M {\bfw})) && :T_{43}.
\end{align*}
We bound $T_{4i}$ separately. 
\begin{align*}
T_{41} &= \calO((\bfu, \bfu); (\delta_{\bfu}, \delta_{\uhat}), ( \itPi_V {\bfw}, \itPi_M {\bfw})) \\
& = ( \delta_{\bfu} \otimes \bfu, \itPi_V \bfw)_{\calT_h} + \langle \tau_C(\bfu) (\delta_{\bfu} - \delta_{\uhat}) , \itPi_V \bfw - \itPi_M \bfw \rangle_{\partial \calT_h} + \langle (\delta_{\uhat} \otimes \bfu) \bfn, \itPi_V \bfw - \itPi_M \bfw \rangle_{\partial \calT_h} \\
&\le \|\bfu\|_{\infty, \Omega} \|\delta_{\bfu}\|_{\Omega} \|\itPi_V {\bfw}\|_{\Omega} + \|\bfu\|_{\infty, \Omega} \; h^{\frac12} (\|\delta_{\bfu} - \delta_{\uhat}\|_{\partial \calT_h} + \|\delta_{\uhat}\|_{\partial \calT_h}) \; h^{-\frac12} \|\itPi_V \bfw - \itPi_M \bfw\|_{\partial \calT_h} \\
& \le C h^{k+1} \|\bfu\|_{\infty, \Omega} \|\bfu\|_{k+1, \Omega} \nrm (\itPi_V \bfw, \itPi_M \bfw) \nrm_{1,h} \le C h^{k+1} \|\bfu\|_{\infty, \Omega} \|\bfw\|_{1, \Omega},
\end{align*}
the last step is by the approximation properties of the projections \eqref{basic_5}, \eqref{basic_6}.

For $T_{43}$, due to the linearity of $\mathcal{O}$ on the last two components and \eqref{O_estimate1}, we have
\begin{align*}
T_{43} &= \calO((\bfu_h, \uhat_h); (e_{\bfu}, e_{\uhat}),  ( \itPi_V {\bfw}, \itPi_M {\bfw})) \le C \nrm (\bfu_h, \uhat_h)\nrm_{1,h} \; \nrm (e_{\bfu}, e_{\uhat}) \nrm_{1,h} \; \nrm ( \itPi_V {\bfw}, \itPi_M {\bfw}) \nrm_{1,h}  \\
& \le C (\nrm (\itPi_V \bfu, \itPi_M \bfu)\nrm_{1,h} + \nrm (e_{\bfu}, e_{\uhat}) \nrm_{1,h}) \; \nrm (e_{\bfu}, e_{\uhat}) \nrm_{1,h} \; \nrm ( \itPi_V {\bfw}, \itPi_M {\bfw}) \nrm_{1,h} \\
& \le C h^{k+1} \|\bfu\|_{1, \Omega} \nrm ( \itPi_V {\bfw}, \itPi_M {\bfw}) \nrm_{1,h} \le C h^{k+1} \|\bfw\|_{1, \Omega},
\end{align*}
by \eqref{basic_7} and Lemma \ref{estimate_EL}. 

For $T_{42}$, if we directly apply \eqref{basic_7}, we will only obtain suboptimal convergence rates. Alternatively, we need a refined analysis for this term. First, we let $E_{\bfu}:= \bfu - \bfu_h = e_{\bfu} + \delta_{\bfu}$ and $E_{\uhat}:= \bfu - \uhat = e_{\uhat} + \delta_{\uhat}$. Next, by the definition of $\mathcal{O}$, we can write $T_{42}$ as 
\begin{align*}
T_{42} & = -(\itPi_V \bfu \otimes E_{\bfu}, \nabla \itPi_V \bfw)_{\calT_h} - (\frac12 (\nabla \cdot E_{\bfu}) \itPi_V \bfu, \itPi_V \bfw)_{\calT_h} + \langle \frac12 (\itPi_V \bfu \otimes (E_{\bfu} - E_{\uhat})) \bfn , \itPi_V \bfw \rangle_{\partial \calT_h} \\
&+ \langle (\tau_C(\bfu) -\tau_C(\uhat_h)) (\itPi_V \bfu - \itPi_M \bfu), \itPi_V \bfw - \itPi_M \bfw \rangle_{\partial \calT_h} - 
\langle  (\itPi_M \bfu \otimes E_{\uhat}) \bfn, \itPi_V \bfw - \itPi_M \bfw \rangle_{\partial \calT_h} \\
& = S_1 + \dots + S_5.
\end{align*}
Notice that by Lemma \ref{estimate_EL} and \eqref{basic_6} we have
\begin{subequations}
\begin{align}\label{estimate_EU}
\|E_{\bfu}\|_{1, h} &\le \|e_{\bfu}\|_{1, h} + \|\delta_{\bfu}\|_{1, h} \le C h^{k+1} \|\bfu\|_{k+2, \Omega}, \\
\label{estimate_EUhat}
\|E_{\uhat}\|_{\partial \calT_h} &\le \|\delta_{\uhat}\|_{\partial \calT_h} + \|e_{\bfu}\|_{\partial \calT_h} + \|e_{\bfu} - e_{\uhat}\|_{\partial \calT_h} \le C h^{k+\frac12},
\end{align}
\end{subequations}
by \eqref{basic_PM}, Lemma \ref{estimate_EL}.
Now we bound each of $S_i$. By the generalized H\"{o}lder's inequality, we have
\[
S_1 \le \|\itPi_V \bfu\|_{\infty, \Omega} \|E_{\bfu}\|_{\Omega} \|\nabla \itPi_V \bfw\|_{\Omega} \le C h^{k+1} \|\bfu\|_{\infty, \Omega} \|\bfw \|_{1, \Omega}.
\]
By a similar argument, we can bound $S_2$ as:
\[S_2 \le C h^{k+1} \|\bfu\|_{\infty, \Omega} \|\bfw \|_{1, \Omega}.\]
For $S_4$, we apply the generalized H\"{o}lder's inequality to have
\begin{align*}
S_4 &\le h^{\frac12}\|\tau_C(\bfu) -\tau_C(\uhat_h)\|_{\partial \calT_h} \| \itPi_V \bfu - \itPi_M \bfu \|_{\infty, \partial \calT_h} \; h^{-\frac12} \| \itPi_V \bfw - \itPi_M \bfw\|_{\partial \calT_h} \\
&\le C h^{\frac12}\| \bfu - \uhat_h \|_{\partial \calT_h} \|\bfu\|_{\infty, \Omega} \; h^{-\frac12} \| \itPi_V \bfw - \itPi_M \bfw\|_{\partial \calT_h} \\
& \le C h^{k+1} \|\bfu\|_{\infty, \Omega} \|\bfw \|_{1, \Omega},
\end{align*}
by the estimate \eqref{estimate_EUhat} and \eqref{basic_7}.

By a similar argument we can bound $S_5$ as
\[
S_5 \le C h^{k+1} \|\bfu\|_{\infty, \Omega} \|\bfw\|_{1, \Omega}.
\]
For the last term $S_3$, if we apply similar estimate as the others, we will only obtain suboptimal order convergence rates. Therefore, we need a refined estimate for this term. We rewrite $S_3$ as follows:

\begin{align*}
S_3 &= \langle \frac12 (\itPi_V \bfu \otimes (e_{\bfu} - e_{\uhat})) \bfn , \itPi_V \bfw \rangle_{\partial \calT_h} + \langle \frac12 (\itPi_V \bfu \otimes \delta_{\bfu}) \bfn , \itPi_V \bfw \rangle_{\partial \calT_h} - \langle \frac12 (\itPi_V \bfu \otimes \delta_{\uhat}) \bfn , \itPi_V \bfw \rangle_{\partial \calT_h} \\
& \le h^{\frac12} \|\itPi_V \bfw\|_{\partial \calT_h} \|\itPi_V \bfu\|_{\infty, \partial \calT_h} h^{-\frac12}(\|e_{\bfu} - e_{\uhat}\|_{\partial \calT_h} + \|\delta_{\bfu}\|_{\partial \calT_h} ) - \langle \frac12 (\itPi_V \bfu \otimes \delta_{\uhat}) \bfn , \itPi_V \bfw \rangle_{\partial \calT_h} \\
&\le C h^{k+1} - \frac12 \langle (\itPi_V \bfu \otimes \delta_{\uhat}) \bfn , \itPi_V \bfw \rangle_{\partial \calT_h},
\end{align*}
by \eqref{basic_6a}, Lemma \ref{estimate_EL}. For the last term, we further split it into two terms as:
\begin{align*}
\langle (\itPi_V \bfu \otimes \delta_{\uhat}) \bfn , \itPi_V \bfw \rangle_{\partial \calT_h} &= \langle (\itPi_V \bfu \otimes \delta_{\uhat}) \bfn , \bfw \rangle_{\partial \calT_h} - \langle (\itPi_V \bfu \otimes \delta_{\uhat}) \bfn , \bfw - \itPi_V \bfw \rangle_{\partial \calT_h} \\
& = - \langle (( \bfu - \itPi_V \bfu) \otimes \delta_{\uhat}) \bfn , \bfw \rangle_{\partial \calT_h} - \langle (\itPi_V \bfu \otimes \delta_{\uhat}) \bfn , \bfw - \itPi_V \bfw \rangle_{\partial \calT_h}, \\
\intertext{
where the last step is obtained by inserting a zero term $\langle ( \bfu  \otimes \delta_{\uhat}) \bfn , \bfw \rangle_{\partial \calT_h} = \langle ( \bfu  \otimes \delta_{\uhat}) \bfn , \bfw \rangle_{\partial \Omega} = 0$.}
&\le h^{\frac14}\|\delta_{\bfu}\|_{L^4(\partial \calT_h)} \; h^{\frac14} \|\bfw\|_{L^4(\partial \calT_h)} \; h^{-\frac12}\|\delta_{\uhat}\|_{\partial \calT_h} + \|\itPi_V \bfu \|_{\infty, \partial \calT_h} \|\delta_{\uhat}\|_{\partial \calT_h} \|\bfw - \itPi_V \bfw\|_{\partial \calT_h}\\
&\le C h^k \|\delta_{\bfu}\|_{1, h} \|\bfw\|_{1, \Omega} + C h^{k+1} \|\bfu\|_{\infty, \Omega} \|\bfw\|_{1, \Omega} \\
& \le C h^{k+1} (\|\bfu\|_{2, \Omega} + \|\bfu\|_{\infty, \Omega}) \|\bfw\|_{1, \Omega},
\end{align*}
where in the last step we used the inequalities \eqref{basic_3}, \eqref{basic_6}, \eqref{basic_PM}. The proof is complete if we combine all the above estimates.
\end{proof}

{\bf Step 4: Optimal estimate for $e_{\bfu}$.} Notice that Lemma \ref{estimate_EL} provides an optimal estimate for $e_{\Lm}$ but only suboptimal estimate for $e_{\bfu}$. This is due to the fact that we use a $P_{k+1}$ polynomial space for the unknown $\bfu$. To obtain an optimal convergence estimate for $e_{\bfu}$ we will use the adjoint problem \eqref{eq:dualNSE} to apply a duality argument. We begin by the following identity for the error $e_{\bfu}$:

\begin{lemma}\label{eu_identity}
Let $(\bfphi,\psi)$ be the solution of the dual problem \eqref{eq:dualNSE} with the source term $\bftheta = e_{\bfu}$, then we have
\begin{align*}
\|e_{\bfu}\|^2_{\Omega}  = & - \langle e_{\bfu} - e_{\uhat}, \nu \delta_{\Phim} \bfn + \delta_{\psi} \bfn \rangle_{\partial \calT_h} \\
& - \langle \frac{\nu}{h} (\itPi_M e_{\bfu} - e_{\uhat}), \itPi_V \bfphi - \itPi_M \bfphi \rangle_{\partial \calT_h} \\
& + \langle \nu \delta_{\Lm} \bfn - \delta_p \bfn - \frac{\nu}{h} \itPi_M \delta_{\bfu}, \itPi_V \bfphi - \itPi_M \bfphi \rangle_{\partial \calT_h} \\
& - \Big( (e_{\bfu}, \nabla \cdot ( \bfphi \otimes  \bfu))_{\calT_h} + \calO((\bfu, \bfu); (e_{\bfu}, e_{\uhat}), (\itPi_V \bfphi, \itPi_M \bfphi)) \Big) \\
& - \calO((\bfu, \bfu); (\delta_{\bfu}, \delta_{\uhat}), (\itPi_V \bfphi, \itPi_M \bfphi)) \\
& + \Big( \calO((\bfu, \bfu); (\bfu_h, \uhat_h), (\delta_{\bfphi}, \delta_{\widehat{\bfphi}})) - \calO((\bfu_h, \uhat_h); (\bfu_h, \uhat_h), 
(\delta_{\bfphi}, \delta_{\widehat{\bfphi}})) \Big) \\
& + \Big( \calO((\bfu_h, \uhat_h); (\bfu_h, \uhat_h), (\bfphi, \bfphi)) - \calO((\bfu, \uhat); (\bfu_h, \uhat_h), (\bfphi, \bfphi)) - (e_{\bfu}, \bfY)_{\calT_h} \Big) \\
& + (\nu e_{\Lm}, \nabla \delta_{\bfphi})_{\calT_h} - (e_p, \nabla \cdot \delta_{\bfphi})_{\calT_h}  \\
& := T_1 + \cdots + T_8.
\end{align*}
Here $\bfY := \frac12 (\nabla \bfphi)^{\top}\bfu - \frac12 (\nabla \bfu)^{\top} \bfphi$ and
\[
\delta_{\Phim} = \Phim - \itPi_G \Phim, \quad  \delta_{\bfphi} := \bfphi - \itPi_V \bfphi, \quad \delta_{\psi} := \psi - \itPi_Q \psi, 
\quad \delta_{\phihat} := \bfphi - \itPi_M \bfphi.
\]
\end{lemma}

\begin{proof}
By the adjoint problem \eqref{eq:dualNSE1} - \eqref{eq:dualNSE3} we have 
\begin{align*}
\|e_{\bfu}\|^2_{\Omega}  = & - \nu (e_{\bfu}, \nabla \cdot \Phim)_{\calT_h} - (e_{\bfu}, \nabla \cdot ( \bfphi \otimes  \bfu)_{\calT_h} 
- (e_{\bfu}, \nabla \psi)_{\calT_h} - (e_{\bfu}, \bfY)_{\calT_h} \\
& - (\nu e_{\Lm}, \Phim)_{\calT_h} + (\nu e_{\Lm}, \nabla \bfphi)_{\calT_h} \\
& - (e_p, \nabla \cdot \bfphi)_{\calT_h}\\
\intertext{rearranging the terms, we have}
= & - \nu (e_{\bfu}, \nabla \cdot \Phim)_{\calT_h} - (\nu e_{\Lm}, \Phim)_{\calT_h}     \\
& - (e_{\bfu}, \nabla \psi)_{\calT_h} \\
& - (e_{\bfu}, \nabla \cdot ( \bfphi \otimes  \bfu)_{\calT_h} + (\nu e_{\Lm}, \nabla \bfphi)_{\calT_h} - (e_p, \nabla \cdot \bfphi)_{\calT_h} - (e_{\bfu}, \bfY)_{\calT_h} \\
= & - \nu (e_{\bfu}, \nabla \cdot \itPi_G \Phim)_{\calT_h} - (\nu e_{\Lm}, \itPi_G \Phim)_{\calT_h} - \nu (e_{\bfu}, \nabla \cdot \delta_{\Phim})_{\calT_h}    \\
& - (e_{\bfu}, \nabla \itPi_Q \psi)_{\calT_h} - (e_{\bfu}, \nabla \delta_{\psi})_{\calT_h}\\
& - (e_{\bfu}, \nabla \cdot ( \bfphi \otimes  \bfu)_{\calT_h} + (\nu e_{\Lm}, \nabla \bfphi)_{\calT_h} - (e_p, \nabla \cdot \bfphi)_{\calT_h} - (e_{\bfu}, \bfY)_{\calT_h}\\
\intertext{taking $(\Gm,\bfv, q,\bfmu) = (\nu \itPi_G \Phim, 0, \itPi_Q \psi , 0)$ in the error equation \eqref{eq:error}, inserting the resulting identity into 
the above expression and simplifying, we have}
= & - \langle e_{\uhat}, \nu \itPi_G \Phim \bfn \rangle_{\partial\calT_h} - \nu (e_{\bfu}, \nabla \cdot \delta_{\Phim})_{\calT_h}    \\
& - \langle e_{\uhat},  \itPi_Q \psi \bfn \rangle_{\partial\calT_h} - (e_{\bfu}, \nabla \delta_{\psi})_{\calT_h}\\
& - (e_{\bfu}, \nabla \cdot ( \bfphi \otimes  \bfu)_{\calT_h} + (\nu e_{\Lm}, \nabla \bfphi)_{\calT_h} - (e_p, \nabla \cdot \bfphi)_{\calT_h} - (e_{\bfu}, \bfY)_{\calT_h} \\
\intertext{inserting two zero terms: $\langle e_{\uhat}, \nu \Phim \bfn \rangle_{\partial \calT_h} = \langle e_{\uhat}, \psi \bfn \rangle_{\partial \calT_h}$ 
and integrating by parts in the first two lines to obtain}
= & - \langle e_{\bfu} - e_{\uhat}, \nu \delta_{\Phim} \bfn + \delta_{\psi} \bfn \rangle_{\partial \calT_h} \\
& - (e_{\bfu}, \nabla \cdot ( \bfphi \otimes  \bfu)_{\calT_h} + (\nu e_{\Lm}, \nabla \bfphi)_{\calT_h} - (e_p, \nabla \cdot \bfphi)_{\calT_h} - (e_{\bfu}, \bfY)_{\calT_h}.
\end{align*}
Next we work on the last line in the above expression. We first insert the projection of $\bfphi$ to have
\begin{align*}
 - (e_{\bfu},  \nabla \cdot ( \bfphi & \otimes  \bfu)_{\calT_h}  + (\nu e_{\Lm}, \nabla \bfphi)_{\calT_h} 
 - (e_p, \nabla \cdot \bfphi)_{\calT_h} - (e_{\bfu}, \bfY)_{\calT_h} \\
& =  (\nu e_{\Lm}, \nabla \itPi_V \bfphi)_{\calT_h} - (e_p, \nabla \cdot \itPi_V \bfphi)_{\calT_h} 
- (e_{\bfu}, \nabla \cdot ( \bfphi \otimes  \bfu)_{\calT_h} - (e_{\bfu}, \bfY)_{\calT_h} \\
& + (\nu e_{\Lm}, \nabla \delta_{\bfphi})_{\calT_h} - (e_p, \nabla \cdot \delta_{\bfphi})_{\calT_h}  \\
\intertext{taking $(\Gm,\bfv, q,\bfmu) = (0, \itPi_V \bfphi, 0 , \itPi_M \bfphi)$ in the error equation \eqref{eq:error}, 
intergrating by parts for the last two terms in the above expression and simplifying, we have,}
& = - \langle \frac{\nu}{h} (\itPi_M e_{\bfu} - e_{\uhat}), \itPi_V \bfphi - \itPi_M \bfphi \rangle_{\partial \calT_h} 
+ \langle \nu \delta_{\Lm} \bfn - \delta_p \bfn - \frac{\nu}{h} \itPi_M \delta_{\bfu}, \itPi_V \bfphi - \itPi_M \bfphi \rangle_{\partial \calT_h} \\
& \quad - \calO((\bfu, \bfu); (\bfu, \bfu), (\itPi_V \bfphi, \itPi_M \bfphi)) + \calO((\bfu_h, \uhat_h); (\bfu_h, \uhat_h), (\itPi_V \bfphi, \itPi_M \bfphi)) \\
& \quad - (e_{\bfu}, \nabla \cdot ( \bfphi \otimes  \bfu)_{\calT_h} - (e_{\bfu}, \bfY)_{\calT_h} \\
& \quad + (\nu e_{\Lm}, \nabla \delta_{\bfphi})_{\calT_h} - (e_p, \nabla \cdot \delta_{\bfphi})_{\calT_h}  \\
& = - \langle \frac{\nu}{h} (\itPi_M e_{\bfu} - e_{\uhat}), \itPi_V \bfphi - \itPi_M \bfphi \rangle_{\partial \calT_h} 
+ \langle \nu \delta_{\Lm} \bfn - \delta_p \bfn - \frac{\nu}{h} \itPi_M \delta_{\bfu}, \itPi_V \bfphi - \itPi_M \bfphi \rangle_{\partial \calT_h} \\
& \quad - \Big( (e_{\bfu}, \nabla \cdot ( \bfphi \otimes  \bfu)_{\calT_h} + \calO((\bfu, \bfu); (e_{\bfu}, e_{\uhat}), (\itPi_V \bfphi, \itPi_M \bfphi)) \Big) \\
& \quad - \calO((\bfu, \bfu); (\delta_{\bfu}, \delta_{\uhat}), (\itPi_V \bfphi, \itPi_M \bfphi)) \\
& \quad + \calO((\bfu_h, \uhat_h); (\bfu_h, \uhat_h), (\itPi_V \bfphi, \itPi_M \bfphi)) - \calO((\bfu, \uhat); (\bfu_h, \uhat_h), 
(\itPi_V \bfphi, \itPi_M \bfphi)) - (e_{\bfu}, \bfY)_{\calT_h}\\
& \quad + (\nu e_{\Lm}, \nabla \delta_{\bfphi})_{\calT_h} - (e_p, \nabla \cdot \delta_{\bfphi})_{\calT_h}. 
\end{align*}
We can obtain the expression in the Lemma by inserting $(\bfphi, \bfphi)$ in the two $\mathcal{O}$ terms in the above identity. This completes the proof. 
\end{proof}

Now we are ready to prove our last result:

\begin{lemma}\label{estimate_eu}
Under the same assumption as in Lemma \ref{estimate_EL}, in addition we assume the full $H^2-$regularity of the adjoint problem \eqref{eq:dualreg} holds 
and $k\geq 1$, then we have
\[
\|e_{\bfu}\|_{\Omega} \le \mathcal{C} h^{k+2},
\]
Here the constant $\mathcal{C}$ depends on $\|\bfu\|_{k+2, \Omega}, \|\bfu\|_{W^{1, \infty}(\Omega)}, \|p\|_{k+1, \Omega}, \nu$ and $k$ but independent of $h$. 
\end{lemma}

\begin{proof}
By identity in Lemma \ref{eu_identity}, it suffice to estimate $T_{1} - T_{8}$. 

For $T_1$, we apply Cauchy-Schwarz inequality, Lemma \ref{estimate_EL}, \eqref{basic_6a} and the regularity inequality \eqref{eq:dualreg} to have
\[
T_1 \le h^{-\frac12}\|e_{\bfu} - e_{\uhat}\|_{\partial \calT_h} h^{\frac12} \|\nu \delta_{\Phim} \bfn + \delta_{\phi} \bfn\|_{\partial \calT_h} \le C h^{k+1} \cdot h^{\frac12} h^{\frac12} (\|\Phim\|_{1, \Omega} + \|\phi\|_{1, \Omega}) \le \mathcal{C} h^{k+2} \|e_{\bfu}\|_{\Omega}.
\]

Similarly, for $T_2$ we have
\[
T_2 \le \nu h^{-\frac12} \|\itPi_M e_{\bfu} - e_{\uhat}\|_{\partial \calT_h} h^{-\frac12} \|\itPi_V \bfphi - \itPi_M \bfphi\|_{\partial \calT_h} \le C h^{k+1} \cdot h^{-\frac12} h^{\frac32} \|\bfphi\|_{2, \Omega} \le \mathcal{C} h^{k+2} \|e_{\bfu}\|_{\Omega}.
\]
Using Cauchy-Schwarz inequality, \eqref{basic_6a}, \eqref{basic_PM} and \eqref{eq:dualreg}, we can bound $T_3$ as
\[
T_3 \le C h^{k+\frac12} (\|\Lm\|_{k+1, \Omega} + \|p\|_{k+1, \Omega} + \|\bfu\|_{k+2, \Omega}) \, h^{\frac32}\|\bfphi\|_{2, \Omega} \le \mathcal{C} h^{k+2} \|e_{\bfu}\|_{\Omega}.
\]
For $T_8$, we simply apply the Cauchy-Schwarz inequality, \eqref{basic_6}, Lemma \ref{estimate_EL}, Lemma \ref{estimate_ep} and the regularity inequality  \eqref{eq:dualreg} to have
\[
T_8 \le (\nu \|e_{\Lm}\|_{\Omega} + \|e_p\|_{\Omega}) \|\nabla \delta_{\phi}\|_{\calT_h} \le C h  (\nu \|e_{\Lm}\|_{\Omega} + \|e_p\|_{\Omega}) \| \bfphi\|_{2, \Omega} \le \mathcal{C} h^{k+2} \|e_{\bfu}\|_{\Omega}. 
\]
For $T_5$, we explicitly write this term:
\begin{align*}
T_5 & = - ( \delta_{\bfu} \otimes \bfu, \nabla \itPi_V \bfphi)_{\calT_h} + \langle \tau_C(\bfu) (\itPi_M \bfu - \itPi_V \bfu), \itPi_V \bfphi - \itPi_M \bfphi \rangle_{\partial \calT_h} + \langle  (\bfu \otimes \delta_{\uhat}) \bfn, \itPi_V \bfphi - \itPi_M \bfphi \rangle_{\partial \calT_h} \\
& \le \|\bfu\|_{\infty, \Omega} \Big( \|\delta_{\bfu}\|_{\calT_h} \|\nabla \itPi_V \bfphi\|_{\calT_h} + \|\itPi_M \bfu - \itPi_V \bfu\|_{\partial \calT_h} \|\itPi_M \bfphi - \itPi_V \bfphi\|_{\partial \calT_h} + \|\delta_{\uhat}\|_{\partial \calT_h}  \|\itPi_M \bfphi - \itPi_V \bfphi\|_{\partial \calT_h} \Big) \\
& \le C \|\bfu\|_{\infty, \Omega} ( h^{k+2} \|\bfu\|_{k+2, \Omega} \|\bfphi\|_{1, \Omega} + h^{k + \frac12} \|\bfu\|_{k+1, \Omega} h^{\frac32} \|\bfphi\|_{2, \Omega} ) \quad \text{by \eqref{basic_6}, \eqref{basic_6a} and \eqref{basic_PM} } \\
& \le \mathcal{C} h^{k+2} \|e_{\bfu}\|_{\Omega},
\end{align*}
by the regularity assumption \eqref{eq:dualreg}.

For $T_4$, we first expand the term as:
\begin{align*}
T_4 = & - (e_{\bfu}, (\nabla \cdot \bfu) \bfphi)_{\calT_h} - (e_{\bfu} \otimes \bfu, \nabla \bfphi)_{\calT_h} + (e_{\bfu} \otimes \bfu, \nabla \itPi_V \bfphi)_{\calT_h} \\
& - \langle \tau_C(\bfu) (e_{\bfu} - e_{\uhat}), \itPi_V \bfphi - \itPi_M \bfphi \rangle_{\partial \calT_h} - \langle (e_{\uhat} \otimes \bfu) \bfn, \itPi_V \bfphi - \itPi_M \bfphi \rangle_{\partial \calT_h} \\
=& - (e_{\bfu} \otimes \bfu, \nabla \delta_{\bfphi})_{\calT_h}  - \langle \tau_C(\bfu) (e_{\bfu} - e_{\uhat}), \itPi_V \bfphi - \itPi_M \bfphi \rangle_{\partial \calT_h} - \langle (e_{\uhat} \otimes \bfu) \bfn, \itPi_V \bfphi - \itPi_M \bfphi \rangle_{\partial \calT_h}\\
\le & C \|\bfu\|_{\infty, \Omega} ( \|e_{\bfu}\|_{\Omega} \|\nabla \delta_{\bfphi}\|_{\calT_h} + h^{-\frac12}\|e_{\bfu} - e_{\uhat}\|_{\partial \calT_h} h^{\frac12} \|\itPi_V \bfphi - \itPi_M \bfphi\|_{\partial \calT_h} + \| e_{\uhat}\|_{\partial \calT_h} \|\itPi_V \bfphi - \itPi_M \bfphi\|_{\partial \calT_h}) \\
\le & \mathcal{C} h^{k+2} \|e_{\bfu}\|_{\Omega} + C h^{\frac32} \|e_{\bfu}\|_{\Omega} \|e_{\uhat}\|_{\partial \calT_h},
\end{align*}
by Lemma \ref{estimate_EL}, \eqref{basic_6}, \eqref{basic_6a} and \eqref{basic_PM}. By a triangle inequality we have
\[
\|e_{\uhat}\|_{\partial \calT_h} \le \|e_{\bfu} - e_{\uhat}\|_{\partial \calT_h} + \|e_{\bfu}\|_{\partial \calT_h} \le \mathcal{C} (h^{k+\frac32} + h^{k + \frac12}).
\]
Inserting this inequality into the estimate for $T_4$ we obtain:
\[
T_4 \le \mathcal{C} h^{k+2} \|e_{\bfu}\|_{\Omega}.
\]

To bound $T_6$, we first derive some useful inequalities, we first bound $\|\bfu_h\|_{\infty, \Omega}$:
\begin{equation}\label{uhinfty}
\|\bfu_h\|_{\infty, \Omega} \le \|e_{\bfu}\|_{\infty, \Omega} + \|\itPi_V \bfu\|_{\infty, \Omega} \le C (h^{-\frac{d}{2}} \|e_{\bfu}\|_{\Omega} + \|\bfu\|_{\infty, \Omega}).
\end{equation}
Next by a triangle inequality, we have
\begin{equation}\label{u-uhat}
\|\bfu_h - \uhat_h\|_{\partial \calT_h} \le \|e_{\bfu} - e_{\uhat}\|_{\partial \calT_h} + \|\itPi_V \bfu - \itPi_M \bfu\|_{\partial \calT_h} \le \mathcal{C} (h^{k+\frac32} + h^{k+\frac12}) \le \mathcal{C} h^{k+\frac12}. 
\end{equation}
Consequently, we have
\begin{equation}\label{uhat_infty}
\|\uhat_h\|_{\infty,\partial \calT_h} \le \|\bfu_h - \uhat_h\|_{\infty, \partial \calT_h} + \|\bfu_h\|_{\infty, \partial \calT_h} \le \mathcal{C} h^{k+ 1 - \frac{d}{2}} + C (h^{-\frac{d}{2}} \|e_{\bfu}\|_{\Omega} + \|\bfu\|_{\infty, \Omega}).
\end{equation}
The last step we applied a scaling argument for the polynomials on $\partial \calT_h$. Finally, applying triangle inequality we obtain the following estimates:
\begin{subequations}\label{u_uh}
\begin{align}
\|\bfu - \bfu_h\|_{\Omega} &\le \|e_{\bfu}\|_{\Omega} + \|\delta_{\bfu}\|_{\Omega} \le \mathcal{C} h^{k+1}, \\
\|\nabla(\bfu - \bfu_h)\|_{\calT_h} & \le \|\nabla e_{\bfu}\|_{\calT_h} + \|\nabla  \delta_{\bfu}\|_{\calT_h} \le \mathcal{C} h^{k+1}, \\
\|\bfu - \uhat_h\|_{\partial \calT_h} & \le \|\delta_{\uhat}\|_{\partial \calT_h} + \|e_{\uhat}\|_{\partial \calT_h} \le \mathcal{C} h^{k+\frac12}.
\end{align}
\end{subequations}
Now we are ready to present the estimate for $T_6$. If we expand $T_6$ using the definition of $\mathcal{O}$, we obtain:
\begin{align*}
T_6 = & - (\bfu_h \otimes (\bfu - \bfu_h), \nabla \delta_{\bfphi}) - (\frac12 (\nabla \cdot (\bfu -\bfu_h)) \bfu_h, \delta_{\bfphi}) + \langle \frac12 (\bfu_h \otimes (\uhat_h - \bfu_h)) \bfn, \delta_{\bfphi} \rangle_{\partial \calT_h} \\
& + \langle (\tau_C(\bfu) -\tau_C(\uhat_h)) (\uhat_h - \bfu_h), \delta_{\bfphi} - \delta_{\widehat{\bfphi}} \rangle_{\partial \calT_h}+ \langle (\uhat_h \otimes (\bfu - \uhat_h)) \bfn, \delta_{\bfphi} - \delta_{\widehat{\bfphi}} \rangle_{\partial \calT_h} \\
\intertext{applying the generalized H\"{o}lder's inequality for each term, we have}
\le & \|\bfu_h\|_{\infty, \Omega} (\|\bfu - \bfu_h\|_{\Omega} \|\nabla \delta_{\bfphi}\|_{\calT_h} + \|\nabla \cdot (\bfu - \bfu_h)\|_{\calT_h} \|\delta_{\bfphi}\|_{\Omega}) + \|\bfu_h\|_{\infty, \Omega} \|\bfu_h - \uhat_h\|_{\partial \calT_h} \|\delta_{\bfphi}\|_{\partial \calT_h} \\
&+ (\|\bfu\|_{\infty, \Omega} + \|\uhat_h\|_{\infty, \partial \calT_h}) \|\bfu_h - \uhat_h\|_{\partial \calT_h} \|\delta_{\bfphi} - \delta_{\widehat{\bfphi}}\|_{\partial \calT_h} + \|\uhat_h\|_{\infty, \partial \calT_h} \|\bfu - \uhat_h\|_{\partial \calT_h} \|\delta_{\bfphi} - \delta_{\widehat{\bfphi}}\|_{\partial \calT_h}\\
\intertext{now if we apply the inequalities \eqref{uhinfty} - \eqref{u_uh}, \eqref{basic_6} - \eqref{basic_PM} and \eqref{eq:dualreg}, we have}
\le & \mathcal{C} h^{k+2} (h^{k+1 - \frac{d}{2}} + 1) \|e_{\bfu}\|_{\Omega} \le \mathcal{C} h^{k+2} \|e_{\bfu}\|_{\Omega}.
\end{align*}

Finally, we need to estimate $T_7$ which is more involved than the previous terms. To this end, we begin by expanding the nonlinear operator $\mathcal{O}$:
\begin{align*}
T_7 &= (\bfu_h \otimes (\bfu - \bfu_h), \nabla \bfphi)_{\calT_h} + (\frac12 \nabla \cdot (\bfu - \bfu_h) \bfu_h, \bfphi)_{\calT_h} + \langle \frac12 (\bfu_h \otimes (\bfu_h - \uhat_h)) \bfn, \bfphi \rangle_{\partial \calT_h} - (e_{\bfu}, \bfY)_{\calT_h},\\
\intertext{integrating by parts the second term, we have}
&= (\bfu_h \otimes (\bfu - \bfu_h), \nabla \bfphi)_{\calT_h} + \langle \frac12 (\bfu_h \otimes (\bfu - \bfu_h)) \bfn, \bfphi \rangle_{\partial \calT_h} - (\frac12 \bfu_h \otimes (\bfu - \bfu_h), \nabla \bfphi)_{\calT_h} \\
& \quad - (\frac12 \bfphi \otimes (\bfu - \bfu_h), \nabla \bfu_h)_{\calT_h} + \langle \frac12 (\bfu_h \otimes (\bfu_h - \uhat_h)) \bfn, \bfphi \rangle_{\partial \calT_h} - (e_{\bfu}, \bfY)_{\calT_h} \\
& = (\frac12 \bfu_h \otimes (\bfu - \bfu_h), \nabla \bfphi)_{\calT_h} + \langle \frac12 (\bfu_h \otimes (\bfu - \uhat_h)) \bfn, \bfphi \rangle_{\partial \calT_h} - (\frac12 \bfphi \otimes (\bfu - \bfu_h), \nabla \bfu_h)_{\calT_h} - (e_{\bfu}, \bfY)_{\calT_h}, \\
\intertext{inserting the zero term $ - \langle \frac12 (\bfu \otimes (\bfu - \uhat_h)) \bfn, \bfphi \rangle_{\partial \calT_h} = 0$ into above expression,} 
& =- \langle \frac12 ((\bfu - \bfu_h) \otimes (\bfu - \uhat_h)) \bfn, \bfphi \rangle_{\partial \calT_h} + (\frac12 \bfu_h \otimes (\bfu - \bfu_h), \nabla \bfphi)_{\calT_h}  - (\frac12 \bfphi \otimes (\bfu - \bfu_h), \nabla \bfu_h)_{\calT_h} - (e_{\bfu}, \bfY)_{\calT_h} \\
& =- \langle \frac12 ((\bfu - \bfu_h) \otimes (\bfu - \uhat_h)) \bfn, \bfphi \rangle_{\partial \calT_h} - (\frac12( \bfu - \bfu_h) \otimes (\bfu - \bfu_h), \nabla \bfphi)_{\calT_h}  + (\frac12 \bfphi \otimes (\bfu - \bfu_h), \nabla (\bfu - \bfu_h))_{\calT_h} \\
&\quad + (\frac12 \bfu \otimes (\bfu - \bfu_h), \nabla \bfphi)_{\calT_h}  - (\frac12 \bfphi \otimes (\bfu - \bfu_h), \nabla \bfu)_{\calT_h} - (e_{\bfu}, \bfY)_{\calT_h}, \\
\intertext{by the definition of $\bfY = \frac12 (\nabla \bfphi)^{\top}\bfu - \frac12 (\nabla \bfu)^{\top} \bfphi$, we obtain:}
& =- \langle \frac12 ((\bfu - \bfu_h) \otimes (\bfu - \uhat_h)) \bfn, \bfphi \rangle_{\partial \calT_h} - (\frac12( \bfu - \bfu_h) \otimes (\bfu - \bfu_h), \nabla \bfphi)_{\calT_h}  + (\frac12 \bfphi \otimes (\bfu - \bfu_h), \nabla (\bfu - \bfu_h))_{\calT_h} \\
&\quad + (\frac12 \bfu \otimes \delta_{\bfu}, \nabla \bfphi)_{\calT_h}  - (\frac12 \bfphi \otimes \delta_{\bfu}, \nabla \bfu)_{\calT_h} \\
&= T_{71} + \cdots + T_{75}.
\end{align*}
We are going to estimate each of the above terms. For $T_{71}$ we apply the generalized H\"{o}lder's inequality, \eqref{basic_3}, \eqref{u_uh}, and \eqref{eq:dualreg}, 
\[
T_{71} \le \|\bfu - \bfu_h\|_{L^4(\partial \calT_h)} \|\bfu - \uhat_h\|_{\partial \calT_h} \|\bfphi\|_{L^4(\partial \calT_h)} \le C h^{-\frac12} \|\bfu - \bfu_h\|_{1, h} \|\bfphi\|_{1, h}  \|\bfu - \uhat_h\|_{\partial \calT_h} \le \mathcal{C} h^{2k+1} \|e_{\bfu}\|_{\Omega}.
\]
For $T_{72}$, we apply the generalized H\"{o}lder's inequality, \eqref{basic_1}, \eqref{u_uh} and \eqref{eq:dualreg} to get:
\[
T_{72} \le \|\bfu - \bfu_h\|^2_{L^4(\Omega)} \|\nabla \bfphi\|_{\Omega} \le C \|\bfu - \bfu_h\|^2_{1, h} \|\bfphi\|_{1, \Omega} \le \mathcal{C} h^{2k + 2} \|e_{\bfu}\|_{\Omega}.
\]
Similarly, we can bound $T_{73}$ as 
\[
T_{73} \le \|\bfphi\|_{L^4(\Omega)} \|\bfu - \bfu_h\|_{L^4(\Omega)} \|\nabla (\bfu - \bfu_h)\|_{\calT_h} \le \mathcal{C} h^{2k+2} \|e_{\bfu}\|_{\Omega}.
\]
For $T_{74}, T_{75}$ we apply the generalized H\"{o}lder's inequality as
\begin{align*}
T_{74} & \le \|\bfu\|_{\infty, \Omega} \|\delta_{\bfu}\|_{\Omega} \|\nabla \bfphi\|_{\Omega} \le C \|\bfu\|_{\infty, \Omega} h^{k+2} \|e_{\bfu}\|_{\Omega}, \\
T_{75} & \le \|\nabla \bfu\|_{\infty, \Omega} \|\bfphi\|_{\Omega} \|\delta_{\bfu}\|_{\Omega} \le C \|\nabla \bfu\|_{\infty, \Omega} h^{k+2} \|e_{\bfu}\|_{\Omega}.
\end{align*}

The proof is complete by combining all the estimates for $T_1 - T_8$. 
\end{proof}

\section{Concluding remarks} In this paper, we introduced and analyzed a new HDG method for the Navier-Stokes equations. The work can be seen as a continuation of our previous work on HDG methods for linear problems, see \cite{QiuShi2015a, QiuShi2015b}. Comparing with the original HDG method for Navier-Stokes equation \cite{CCQ2015, NPCockburn2011}, our method uses an enriched polynomial space for the velocity in each element, a modified numerical flux and a modified HDG formulation. As a concequence, we obtained optimal order of convergence for all unknowns and superconvergence for the velocity without postprocessing. In addition, similar as in \cite{QiuShi2015a, QiuShi2015b}, the analysis in this paper is valid for general polygonal meshes. 

Numerical study of the method as well as other computational aspects including the characterization of the scheme, implementation of the method using Picard iteration and other related issues will be extensively discussed in a separate paper.

\providecommand{\bysame}{\leavevmode\hbox to3em{\hrulefill}\thinspace}
\providecommand{\MR}{\relax\ifhmode\unskip\space\fi MR }
\providecommand{\MRhref}[2]{%
  \href{http://www.ams.org/mathscinet-getitem?mr=#1}{#2}
}
\providecommand{\href}[2]{#2}

\end{document}